\documentclass[onefignum,onetabnum]{siamart171218}

\usepackage{yhmath}
\usepackage{float}
\usepackage{amsfonts}
\usepackage{amssymb}
\usepackage{graphicx}
\usepackage{epstopdf}
\usepackage{algorithmic}
\ifpdf
  \DeclareGraphicsExtensions{.eps,.pdf,.png,.jpg}
\else
  \DeclareGraphicsExtensions{.eps}
\fi

\usepackage{bm}
\usepackage{tikz-cd}
\usepackage{subfig}
\usepackage{mathtools}
\usepackage{stmaryrd}
\usepackage{multirow}
\usepackage{enumitem}
\usepackage{esint}
\usepackage{booktabs}

\newsiamremark{remark}{Remark}
\newsiamremark{assumption}{Assumption}
\newsiamthm{property}{Property}

\newsiamremark{hypothesis}{Hypothesis}
\crefname{hypothesis}{Hypothesis}{Hypotheses}
\newsiamthm{claim}{Claim}
\newcommand\dx{\mathrm{d}x}
\newcommand\ds{\mathrm{d}s}
\newcommand\T{\mathcal{T}}
\newcommand\F{\mathcal{F}}
\newcommand\N{\mathcal{N}}

\newcommand{\pd}[3]{\frac{\p^{#3} #1}{\p #2^{#3}}}

\newcommand\dphi{\mathrm{d}\varphi}
\newcommand\tr{\operatorname{tr}} 
\newcommand\bo{\boldsymbol}
\newcommand\ubar{\underline}
\newcommand\op{\operatorname}
\newcommand\p{\partial}
\newcommand{\bh}[1]{\bo{\hat{#1}}}
\newcommand\mP{\mathcal{P}}

 \headers{$C^{0}$ Approximation of Oblique Derivative Problem}{G. Gao, S. Wu}

\title{A $C^{0}$ finite element approximation of planar oblique derivative problems in non-divergence form
 \thanks{
 The work of Shuonan Wu is supported in part by the National Natural
 Science Foundation of China grant No.  11901016. }
 }
\author{
Guangwei Gao \quad \quad 
Shuonan Wu\thanks{Corresponding author. School of Mathematical Sciences,
 Peking University, Beijing 100871, China 
.}
}

\begin{document}

 \maketitle
\begin{abstract}
  This paper proposes a $C^{0}$ (non-Lagrange) primal finite element approximation of the linear elliptic equations in non-divergence form with oblique boundary conditions in planar, curved domains. As an extension of [Calcolo, 58 (2022), No. 9], the Miranda-Talenti estimate for oblique boundary conditions at a discrete level is established by enhancing the regularity on the vertices. Consequently, the coercivity constant for the proposed scheme is exactly the same as that from PDE theory. The quasi-optimal order error estimates are established by carefully studying the approximation property of the finite element spaces. Numerical experiments are provided to verify the convergence theory and to demonstrate the accuracy and efficiency of the proposed methods.
\end{abstract}



\section{Introduction}
We consider the elliptic equations in non-divergence form in a planar domain subject to the oblique boundary conditions. 
The model problem is to find $u: \Omega \rightarrow \mathbb{R}$ with $\int_{\Omega} u \dx = 0$ such that  
\begin{equation}\label{eq:obl-der-problem} 
\begin{aligned}
A:D^{2}u &=  f  &\text{a.e. in }\Omega, \\ 
 \ell \cdot \nabla u &= c &\text{on }\partial \Omega.
\end{aligned}
\end{equation} 
The domain $\Omega \subset \mathbb{R}^{2}$ is assumed to have a $C^{2}$ boundary. 
Here, the coefficient matrix $A \in L^{\infty}(\Omega;\mathbb{R}^{2\times 2})$ satisfies the uniform ellipticity condition \eqref{eq:uniform-elliptic}, $\ell$ is a given unit vector field defined on $\p\Omega$ (hereby called the ``oblique vector field"), $A:B := \sum_{i,j = 1}^{2} A_{i,j}B_{i,j}$ denotes the Frobenius inner product of two matrices. 
The constant $c$ appearing in the boundary conditions is intended to absorb the compatibility condition, i.e., $c$ is a priori unknown.  
The problem \eqref{eq:obl-der-problem} considered in this paper corresponds to the homogeneous oblique boundary conditions.

The oblique derivative problems arise from the linearization of transport boundary conditions for the Monge-Amp\'{e}re equation \cite{urbas1997second}.
Further applications of oblique derivative problems include the problem of determining the gravitational fields of celestial bodies \cite{palagachev2008poincare}, and the study of systems of certain conservation laws \cite{keyfitz2000proof, zheng2003global}.

Due to the non-divergence structure of \eqref{eq:obl-der-problem}, the concept of weak solutions based on the integration by part is no longer applicable. Instead, solution concepts such as classical solutions, viscosity solutions, and strong solutions are applicable.
For these different solution concepts, numerical methods for the linear elliptic equations in non-divergence form have recently experienced rapid developments; See \cite{smears2013discontinuous, smears2014discontinuous, smears2014discontinuous, kawecki2019dgfem, brenner2020adaptive, kawecki2020convergence, kawecki2020unified, gallistl2017variational, gallistl2019mixed, qiu2020adaptive, wu2021c0} for approximating the $H^2$ strong solutions, and \cite{feng2017finite,feng2018interior,nochetto2018discrete,lakkis2011finite,lakkis2013finite,wang2018primal,gallistl2019mixed} for others.

This article considers the $H^{2}$ strong solutions of non-divergence form with oblique boundary conditions \eqref{eq:obl-der-problem}. 
Here, the oblique vector field $\ell$ satisfies some technical assumptions (c.f. \cite{maugeri2000elliptic}), which are specified in Assumption \ref{ass:problem}. 
The coefficient matrix $A$ is allowable to be discontinuous. 
As compensation, it is required to satisfy the Cordes condition \eqref{eq:cordes}, which is equivalence to the uniform ellipticity \eqref{eq:uniform-elliptic} in the planar case. In addition, the well-posedness of \eqref{eq:obl-der-problem} hinges on a variant of the Miranda-Talenti estimate.
We refer the reader to \cite{maugeri2000elliptic} for the analysis of PDEs with discontinuous coefficients under the Cordes condition.  


For the numerical approximations of \eqref{eq:obl-der-problem}, a discontinuous Galerkin (DG) method was proposed in \cite{kawecki2019discontinuous}, which is applicable when choosing suitably large penalization parameters. 
A mixed method was proposed in  \cite{gallistl2019numerical}, where the function approximating $\nabla u$ is confined to the piecewise linear space.  
More examples on the numerical apprixmations of oblique derivative problems are discussed in \cite{barrett1985fixed, wen1994finite,favskova2010finite,medl2018numerical}, where \cite{barrett1985fixed, favskova2010finite} apply to a particular geodetic and free boundary problem.

This paper proposed a $C^{0}$ primal finite element approximation of \eqref{eq:obl-der-problem} without introducing any penalization term. 
Following the principle in \cite{wu2021c0}, we established a discrete version of the Miranda-Talenti estimate for the oblique boundary condition by adopting the $C^{0}$ finite element with the enhanced regularity on the vertices. 
A typical family of finite elements that meets this requirement is the
family of $\mathcal{P}_{k}$-Hermite finite elements ($k \geq 3$). 
Since the problem \eqref{eq:obl-der-problem} is solved in a planar $C^{2}$ domain, 
we apply the techniques in \cite{zlamal1973curved,bernardi1989optimal} and 
use the curved Hermite element; See Section \ref{subsec:fem_space} for details.
The jump and boundary terms in the discrete Miranda-Talenti-type identity naturally induce the stabilization term in the proposed numerical scheme. 

Thanks to the Miranda-Talenti estimate at a discrete level, we show the well-posedness of the proposed numerical scheme, which mimics the analysis of $H^{2}$ solutions to a great extent.
A striking feature of the proposed scheme is that the coercivity constant at the discrete level is exactly the same as that from PDE theory. 
Another interesting feature is that the proposed method also gives an approximation of the unknown constant that arises in the compatibility condition.  
The proposed scheme is proved to be consistent, coercive, and bounded. Moreover, we constructed a quasi-interpolation operator that preserves the oblique boundary conditions at a discrete level with optimal order in the energy norm, which naturally leads to the energy norm error estimates.

The remaining parts of this paper are organized as follows. 
Section \ref{sec:pre} reviews the $H^{2}$ strong solution theory of \eqref{eq:obl-der-problem}.
In Section \ref{sec:fem}, we state the finite element spaces and present the proof of the discrete Miranda-Talenti-type identity. 
The numerical scheme for approximating \eqref{eq:obl-der-problem} is proposed and analyzed in Section \ref{sec:analysis}. 
Numerical experiments are presented in Section \ref{sec:numerical}.
Some technical proofs can be found in Appendixes \ref{app:approximation} and \ref{app:poincare}.

We use $D$ to denote a generic subdomain of $\Omega$, and $\p D$ denotes its boundary. 
$W^{s}_{p}(D)$ denotes the standard Sobolev space for $s \geq 0$ and $1 \leq p \leq \infty$, 
$W^{0}_{p}(D) = L^{p}(D)$ and $W^{s}_{2}(D) = H^{s}(D)$. 
$(\cdot,\cdot)_{D}$ denotes the standard inner product on $L^{2}(D)$.  
For convenience, we use $C$ to denote a generic positive constant independent of mesh size $h$. The notation $ X \lesssim Y $ means $ X \leq CY $. $ X \eqsim Y $ means
$ X\lesssim Y $ and $ Y \lesssim X $. We also denote $|\cdot|$ as the Euclidean norm for vectors and the Frobenius norm for matrices.

\section{Review of the \texorpdfstring{$H^2$}{H2} strong solutions} \label{sec:pre}
This section reviews the $H^{2}$ strong solutions to the linear elliptic equations in non-divergence form with the oblique boundary conditions. 

The coefficient $A \in L^{\infty}(\Omega; \mathbb{R}^{2 \times 2})$ is assumed to satisfy the uniform ellipticity, i.e., there exist $\ubar{\nu}, \bar{\nu} > 0$ such that 
\begin{equation} \label{eq:uniform-elliptic}
    \underline{\nu}|\bo{\xi}|^2 \leq \bo{\xi}^tA(x)\bo{\xi} \leq \bar{\nu}
    |\bo{\xi}|^2 \qquad \forall \bo{\xi} \in \mathbb{R}^2, \text{ a.e. in
    }\Omega.
\end{equation} 
It is well known that in two dimensions, the uniform ellipticity \eqref{eq:uniform-elliptic} implies the following Cordes condition \eqref{eq:cordes} with $\varepsilon = 2 \ubar{\nu}\bar{\nu} / (\ubar{\nu}^{2} + \bar{\nu}^{2}) $, see \cite{smears2014discontinuous}.
\begin{definition}[Cordes condition]
  \label{def:cordes}The coefficient satisfies that there is an $\varepsilon \in (0, 1]$ such that 
\begin{equation} \label{eq:cordes}
\frac{|A|^2}{(\tr A)^2} \leq \frac{1}{1 + \varepsilon} 
\qquad \text{a.e. in }\Omega.
\end{equation}    
\end{definition}

Define the strictly positive function $\gamma \in L^{\infty}(\Omega)$ by
$ \gamma := \frac{\tr A}{|A|^{2}}. $
The $H^{2}$ well-posedness of the problem \eqref{eq:obl-der-problem} hinges on the following lemma; see \cite[Lemma 1]{smears2013discontinuous} for the proof. 
\begin{lemma}[property of Cordes condition]
  \label{lm:cordes}Under the Cordes condition \eqref{eq:cordes}, for any $v\in H^2(\Omega)$ and open set $U \subset \Omega$, the following inequality holds a.e. in $U$
\begin{equation} \label{eq:nondiv-Cordes-prop}
|\gamma A:D^{2} v - \Delta v| \leq 
\sqrt{1-\varepsilon} |D^2v|.
\end{equation}
\end{lemma}

The bounded domain $\Omega \subset \mathbb{R}^{2}$ is assumed to have a $C^{2}$ boundary. Moreover, $\partial \Omega$ is parametrized by the arc length $\varphi$, i.e., $\bo{x}: [0,L] \rightarrow \mathbb{R}^{2}$
\begin{equation}\label{eq:omega-para} 
\bo{x}(\varphi) = \begin{pmatrix}
  x_{1}(\varphi) \\
  x_{2}(\varphi) \\
\end{pmatrix} 
\qquad \text{ with } \bo{x}(0) = \bo{x}(L). 
\end{equation}  
For a function $v: [0,L] \rightarrow \mathbb{R}$, let $\dot{v}$ be its derivative with respect to the arc length parameter $\varphi$. Similarly, $\ddot{v}$ denotes the second order derivative.
Note that $\bo{x} \in C^{2}([0,L]; \mathbb{R}^{2})$, 
we denote $\chi(\varphi) := \ddot{x}_{1}(\varphi)\dot{x}_{2}(\varphi) - \ddot{x}_{2}(\varphi)\dot{x}_{1}(\varphi)$ the curvature of $\partial \Omega$ at $\bo{x}(\varphi)$.
Let $n$ be the unit outward normal vector of $\partial \Omega$, and $t = \dot{\bo{x}}$ be the unit tangent vector. 

The oblique vector field $\ell: [0,L] \rightarrow \mathbb{R}^{2}$, which satisfies $\ell(0) = \ell(L)$, is assumed to have $C^{1}$ regularity. We denote $\theta$ the oriented angle (anticlockwise) from $n$ to $\ell$, then $\theta: [0,L] \rightarrow \mathbb{R}$ is of class $C^{1}$. 
We further assume an additional condition on the winding number of $\ell$, namely
\begin{equation}\label{eq:winding-num} 
\frac{\theta(L) - \theta(0)}{2\pi} = 0.
\end{equation} 
This means that $\ell$ does not make a full turn around the normal $n$. 
Note that $\ell$ is defined on the interval $[0,L]$, and it can be identified with a function on $\partial \Omega$. Notation like $\ell(\bo{x}(\varphi))$ instead of $\ell(\varphi)$ will sometimes be used for convenience. 

We define the following subspace of $H^{2}(\Omega)$:
\begin{equation}\label{eq:obl-H2} 
 \begin{aligned}
 H^{2}_{\ell}(\Omega) :&= \{ v \in H^{2}(\Omega): \ell \cdot \nabla v \text{ is constant on }\partial \Omega \}, \\ 
 H^{2}_{\ell,0}(\Omega) &:= \{ v \in H^{2}_{\ell}(\Omega): \int_{\Omega} v \dx = 0\}.
 \end{aligned}
\end{equation} 
The analysis of the $H^{2}$ well-posedness of the problem \eqref{eq:obl-der-problem} hinges on several lemmas introduced below, which extend the important Miranda-Talenti estimate to the case of oblique boundary conditions. The proofs of these lemmas are given in \cite{maugeri2000elliptic}, and we sketch the proof here for completeness.  

\begin{lemma}\label{lm:mt-det} For any $v \in H^{2}_{\ell}(\Omega)$, it holds that
\begin{equation}\label{eq:mt-det} 
\int_{\Omega} (\Delta v)^{2} - |D^{2}v|^{2} \dx = 
 \int^{L}_{0} |\nabla v|^{2}\left(\dot{\theta} - \chi\right) \mathrm{d}\varphi. 
\end{equation} 
\end{lemma} 
\begin{proof}
  First, a direct calculation gives
  \begin{equation*}
    \begin{aligned}
      (\Delta v)^{2} - |D^{2}v|^{2} 
      & = \frac{\partial}{\partial x_{1}}
     \left(\partial_{1}v\partial_{22}v - \partial_{2}v\partial_{12}v\right) - 
     \frac{\partial}{\partial x_{2}}
     \left(\partial_{1}v\partial_{12}v - \partial_{2}v\partial_{11}v\right),
    \end{aligned}
  \end{equation*}
   for sufficiently smooth function $v$. By the divergence Theorem, we obtain 
  $$\begin{aligned}
\int_{\Omega}(\Delta v)^{2} - |D^{2}v|^{2} \dx =
\int_{0}^{L} 
 &\left(\partial_{1}v\partial_{22}v - \partial_{2}v\partial_{12}v\right) n_{1}
 - \left(\partial_{1}v\partial_{12}v - \partial_{2}v\partial_{11}v\right) n_{2}
 \mathrm{d}\varphi.
\end{aligned} $$
Next, we express the first and second order derivative of $v$ as the directional derivative along $\ell$ and its perpendicular direction $\ell^{\perp} := (-\ell_{2}, \ell_{1})$. Using the condition that $\ell \cdot \nabla u$ is constant on the boundary, we can obtain \eqref{eq:mt-det}. We refer to \cite[Page. 51]{maugeri2000elliptic} for more details of the proof. 
\end{proof}

We are now ready to give the Miranda-Talenti estimate in the case of oblique boundary conditions.
\begin{lemma}[Miranda-Talenti estimate]
  \label{lm:mt-obl}Assume the domain $\Omega$ and the oblique vector field $\ell$ satisfy 
\begin{equation}\label{eq:obl-angle} 
 \chi_{0} := \min_{\p\Omega}(\dot{\theta} - \chi) \geq 0.
\end{equation} 
Then for any $v \in H^{2}_{\ell}(\Omega)$, it holds that 
\begin{equation}\label{eq:mt-obl} 
\int_{\Omega} |D^{2}v|^{2}\dx \leq \int_{\Omega} (\Delta v)^{2} \dx. 
\end{equation} 
\end{lemma}
Under a stronger assumption on $\Omega$ and $\ell$, we obtain the following gradient estimate (see \cite[Page. 53]{maugeri2000elliptic}).
\begin{lemma}[gradient estimate]
  \label{lm:grad-esti}Assume the domain $\Omega$ and the oblique vector field $\ell$ satisfy  
  $\chi_{0} = \displaystyle\min_{\p\Omega}(\dot{\theta} - \chi) > 0$.
Then for any $v \in H^{2}_{\ell}(\Omega)$, it holds that
\begin{equation}\label{eq:grad-esti} 
  \int_{\Omega}|\nabla v|^{2}\dx \leq C \int_{\Omega} (\Delta v)^{2} \dx,  
\end{equation} 
where the constant $C$ only depends on $\Omega$ and $\chi_{0}$.
\end{lemma} 
\begin{proof}
  Applying the divergence Theorem to the field 
  $(x_{1}(\partial_{1} v)^{2},x_{2}(\partial_{2} v)^{2})$ and employing the Young inequality, we get the existence of two positive constants $C_{1}$ and $C_{2}$ (depend on $\Omega$), such that
  \begin{equation}\label{eq:grad-esti-1} 
  \int_{\Omega}|\nabla v|^{2}\dx \leq 
  C_{1}\int_{\Omega}|D^{2}v|^{2}\dx + 
  C_{2}\int_{\partial\Omega}|\nabla v|^{2}\ds. 
  \end{equation} 
  For any $A \geq 1$, using Lemma \ref{lm:mt-det}, we have
  \begin{equation}\label{eq:grad-esti-2} 
  \int_{\Omega}|D^{2}v|^{2}\dx \leq A\int_{\Omega}|D^{2}v|^{2}\dx =  
  A\int_{\Omega}(\Delta v)^{2}\dx - 
  A\int_{\partial \Omega}|\nabla v|^{2}(\dot{\theta} - \chi)\ds.
  \end{equation} 
  Substituting \eqref{eq:grad-esti-2} into \eqref{eq:grad-esti-1} and combining it with the condition $\chi_{0} = \min_{\p\Omega}(\dot{\theta} - \chi) >0$, we have
  $$ \int_{\Omega}|\nabla v|^{2}\dx \leq 
  C_{1}A\int_{\Omega}(\Delta v)^{2}\dx + 
  (C_{2} - C_{1}A\chi_{0}) 
  \int_{\partial\Omega}|\nabla v|^{2}\ds.$$
  Then \eqref{eq:grad-esti} is obtained by choosing $A = \max\{ \frac{C_{2}}{C_{1}\chi_{0}},1\}$.
\end{proof}

We find that the semi-norm $|\cdot|_{H^{2}(\Omega)}$ is indeed a norm on $H^{2}_{\ell,0}(\Omega)$ by combining Lemma \ref{lm:grad-esti} (gradient estimate) and the Poincar\'{e} inequality, i.e., for any $v \in H^{2}_{\ell,0}(\Omega)$ 
$$ \begin{aligned}
  \|v\|_{H^{2}(\Omega)}^{2} &= 
  \|v\|_{L^{2}(\Omega)}^{2} + |v|^{2}_{H^{1}(\Omega)} + |v|^{2}_{H^{2}(\Omega)} \\ 
  & \lesssim |v|^{2}_{H^{1}(\Omega)} + |v|^{2}_{H^{2}(\Omega)} ~~~~~~~~~~~~~~~~~(\mbox{Poincar\'{e} inequality})\\
  & \lesssim \|\Delta v\|^{2}_{L^{2}(\Omega)} + |v|^{2}_{H^{2}(\Omega)} \lesssim |v|^{2}_{H^{2}(\Omega)}. ~~~~~~~(\mbox{Lemma \ref{lm:grad-esti}}) 
\end{aligned} $$
Now in the Hilbert space $(H^{2}_{\ell, 0}(\Omega), |\cdot|_{H^{2}(\Omega)})$, we are ready to define the bilinear form $b: H^{2}_{\ell,0}(\Omega) \times H^{2}_{\ell,0}(\Omega) \rightarrow \mathbb{R}$ as 
\begin{equation}\label{eq:bilinear-oblique} 
b(w,v) := \int_{\Omega} \gamma A:D^{2}w \Delta v\mathrm{d}x. 
\end{equation} 
Lemma \ref{lm:cordes} (property of Cordes condition) and Lemma \ref{lm:mt-obl} (Miranda-Talenti estimate) imply the coercivity, i.e., for any $v\in H^{2}_{\ell,0}(\Omega)$
\begin{equation}\label{eq:coer-b} 
  \begin{aligned}
    b(v,v) & = \int_{\Omega} (\Delta v)^{2}\dx + \int_{\Omega} (\gamma A - I):D^{2}v \Delta v \dx  \\ 
    & \geq \|\Delta v\|^{2}_{L^{2}(\Omega)} - \sqrt{1 - \varepsilon}\|D^{2}v\|_{L^{2}(\Omega)}\|\Delta v\|_{L^{2}(\Omega)} 
     \geq (1 - \sqrt{1 - \varepsilon}) |v|_{H^{2}(\Omega)}^{2}.    
  \end{aligned} 
\end{equation} 
The variational form of problem \eqref{eq:obl-der-problem} reads: Find $u \in H^{2}_{\ell, 0}(\Omega)$ such that  
\begin{equation}\label{eq:var-obl-der-problem} 
b(u,v) = l(v) \qquad \forall v \in H^{2}_{\ell, 0}(\Omega),
\end{equation} 
where $l(v) =  \int_{\Omega}\gamma f \Delta v \dx$ is a linear functional on $H^{2}_{\ell,0}(\Omega)$. 
Before claiming the well-posedness result for \eqref{eq:obl-der-problem}, we summarize the assumptions on the data as follows.
\begin{assumption}\label{ass:problem}
  Let $A \in L^{\infty}(\Omega;\mathbb{R}^{2 \times 2})$, satisfy the uniform ellipticity \eqref{eq:uniform-elliptic}. 
  $\Omega$ is assumed to be a bounded domain with $C^{2}$ boundary. 
  The oblique vector field $\ell$ is of class $C^{1}$ with unit length, and \eqref{eq:winding-num} is satisfied.  Furthermore, it is assumed that $\chi_{0} = \displaystyle\min_{\p\Omega}(\dot{\theta} - \chi) > 0$.
\end{assumption} 
\begin{theorem}[well-posedness]
  \label{thm:well-posed-obl}Under Assumption \ref{ass:problem}, for any given $f \in L^{2}(\Omega)$, there exists a unique solution $u \in H^{2}_{\ell,0}(\Omega)$ to the variational problem \eqref{eq:var-obl-der-problem}. Furthermore, $u$ is the $H^{2}$ strong solution to the oblique derivative problem \eqref{eq:obl-der-problem}.
\end{theorem} 
\begin{proof}
  It is easy to verify that the linear form $l(\cdot)$ and the bilinear form $b(\cdot, \cdot)$ are bounded on the Hilbert space $(H^{2}_{\ell,0}(\Omega), |\cdot|_{H^{2}(\Omega)})$. Then the coercivity of $b(\cdot,\cdot)$ and the Lax-Milgram Theorem imply the existence of a unique solution $u \in H^{2}_{\ell,0}(\Omega)$ to the variational problem \eqref{eq:var-obl-der-problem}. 

  Note that for any given $z \in L^{2}(\Omega)$, the condition \eqref{eq:winding-num} implies that there exists a unique $v_{z} \in H^{2}_{\ell,0}(\Omega)$ such that $\Delta v_{z} = z$, a.e. in $\Omega$ (see \cite[Page. 49-50]{maugeri2000elliptic} for details). Therefore, 
  \begin{equation*} 
  \begin{aligned}
  \int_{\Omega}\left(\gamma A:D^{2}u - \gamma f \right) z \dx 
  &=\int_{\Omega}\left(\gamma A:D^{2}u - \gamma f \right) \Delta v_{z} \dx 
  =b(u,v_{z}) - l(v_{z}) = 0. 
  \end{aligned}
  \end{equation*} 
  Then the fundamental Lemma of the calculus of variations yields 
  $\gamma A:D^{2}u - \gamma f = 0$, which leads to 
  $A:D^{2}u = f$, a.e. in $\Omega$ since $\gamma$ is uniformly positive. 
\end{proof}

\section{Finite element space and discrete Miranda-Talenti-type identity}\label{sec:fem}
In this section, we shall construct the finite element space for approximating \eqref{eq:obl-der-problem}. More precisely, we adopt the curved Hermite elements on exact triangulations of the domain $\Omega$.  The discrete Miranda-Talenti-type identity for oblique boundary conditions is therefore established through the $C^{1}$-continuity on the vertices. 

\subsection{Exact triangulation of the curved domain}
Since $\ell$ is only defined on $\p\Omega$, the triangulation needs to strictly fit the curved boundary of the domain $\Omega$. 
For this purpose, the domain $\Omega$ is exactly triangulated by $\T_{h}$, i.e., 
$\T_{h}$ consists of curved triangles near the boundary with at most one truly curved edge (fits the curved boundary).
In the interior of $\Omega$, $\T_{h}$ consists of straight triangles. 
A formal definition of the curved triangle is stated as follows \cite{bernardi1989optimal}. 
\begin{definition}[curved triangle] 
  \label{def:curved-triangle}A closed set $K \subset \mathbb{R}^{2}$ is a curved triangle if there exists a $C^{1}$ mapping $F_{K}$ that maps a straight reference triangle $\hat{K}$ onto $K$ and that is of the form
\begin{equation}\label{eq:FK} 
    F_{K} = \tilde{F}_{K} + \Phi_{K}, 
\end{equation} 
where $\tilde{F}_{K}:\hat{x} \mapsto B_{K} \hat{x} + b_{K}$ is an invertible affine mapping and $\Phi_{K}$ is a $C^{1}$ mapping satisfying
\begin{equation}\label{eq:cK} 
    c_{K} := \displaystyle\sup_{\hat{x} \in \hat{K}} \|D\Phi_{K} B_{K}^{-1}\|  < 1. 
\end{equation} 
\end{definition}


Let $\tilde{K} := \tilde{F}_{K}(\hat{K})$ be a straight triangle that ``approximates" the curved triangle $K$. 
Note that if the mapping $\Phi_{K}$ is equal to zero, $K = \tilde{K}$ is a straight triangle.
The mesh size of $K$ is defined as $h_{K} := \op{diam}(\tilde{K})$. A simple calculation shows that (c.f. \cite[Lemma 3.1]{bernardi1989optimal})
\begin{equation}\label{eq:diamK-hK} 
(1 - c_{K}) h_{K} \leq \op{diam}(K) \leq (1 + c_{K}) h_{K}.
\end{equation}  
Denote $h := \max_{K \in \mathcal{T}_{h}}h_{K}$. 
Let $\mathcal{N}_{h}$ be the set of vertices in $\mathcal{T}_{h}$, $\mathcal{N}_{h}^{\partial} := \mathcal{N}_{h} \cap \partial \Omega$ and $\mathcal{N}^{i}_{h} := \mathcal{N}_{h} \backslash \mathcal{N}^{\partial}_{h}$. 
Denote $\mathcal{F}_{h}$ the set of edges in $\mathcal{T}_{h}$. Let $\mathcal{F}^{\partial}_{h} := \mathcal{F}_{h} \cap \partial \Omega$ and $\mathcal{F}^{i}_{h} := \mathcal{F}_{h} \backslash \mathcal{F}^{\partial}_{h}$. 
For each edge $F \in \F_{h}^{i}$, we define $h_{F} := \min (h_{K}, h_{K^{\prime}})$ where $K$ and $K^{\prime}$ are such that $F = \p K\cap\p K^{\prime}$. If $F \in \F^{\p}_{h}$, we define $h_{F} := h_{K}$ where $K$ is such that $F = \p K\cap\p\Omega$.

A family of conforming triangulations $\{\T_{h}\}_{h>0}$ is said to be shape-regular if there exist two constants $\sigma$ and $c_{1}$ independent of $h$ such that, 
  $$ \sup_{h}\sup_{K\in\T_{h}}\frac{h_{K}}{\rho_{K}} \leq \sigma 
  \quad \text{ and } \quad 
  \sup_{h} \sup_{K \in \mathcal{T}_{h}} c_{K} \leq c_{1} < 1. 
  $$
Here, $\rho_{K}$ is the diameter of the sphere inscribed in $\tilde{K}$.
In the standard finite element theory, the Sobolev space $W^{m}_{p}(K)$ is usually related to $W^{m}_{p}(\hat{K})$ by the push-forward mapping, i.e.,  $ \hat{v} = F_{K}^{*} v = v\circ F_{K}$.
To reproduce this in the curved triangle case, we need to assume some smoothness on $F_{K}$.  
\begin{definition}[class $m$ triangulation] \label{def:triangulation-m}
  The family $\{\T_{h}\}_{h>0}$ is said to be regular of order $m$ if it is shape-regular and if, for each $h$, any $K \in \T_{h}$ the mapping $F_{K}$ is of class $C^{m+1}$, with 
  $$ \sup_{h} \sup_{K \in \T_{h}} c_{i}(K) \leq c_{i} < \infty, \qquad 2 \leq i \leq m+1, $$
where 
$ c_{i}(K) := |F_{K}|_{W^{i}_{\infty}(\hat{K};\mathbb{R}^{2})}\|B_{K}\|^{-i}.$
\end{definition} 
\begin{remark}[construction of $F_{K}$]
  \label{rem:construction-F_K}The construction of $F_{K}$ is specified in \cite[Section 6]{bernardi1989optimal}. 
  Here, we emphasize that the construction of $F_{K}$ guarantees the following fact: if $F$ is a straight edge of the curved triangle $K$, then for any edge $\hat{F}$ of $\hat{K}$, 
    $$ \text{$F_{K}|_{\hat{F}}: \hat{F} \rightarrow F$ is an affine mapping.} $$
 Further, to construct $\{\T_{h}\}_{h>0}$ with order $m$, the 
  mapping $F_{K}$ needs to be $C^{m+1}$, which requires the boundary $\partial \Omega$ to be piecewise $C^{m+1}$. 
\end{remark}

\subsection{Finite element space}\label{subsec:fem_space}
Following the idea in \cite{wu2021c0}, the key to designing the finite element space is the $C^{1}$-continuity on the vertices. 
Since the discretization is based on the exact triangulation of the domain $\Omega$, we adopt the technique in \cite{bernardi1989optimal} to consider the curved Hermite element.
To this end, we start with the standard Hermite element on straight triangles.
\paragraph{\underline{\textit{Hermite element on the straight reference triangle}}}For the straight reference triangle $\hat{K}$, the shape function space is given as $\mathcal{P}_{k}(\hat{K})$ ($k\geq 3$), where $\mathcal{P}_{k}(\hat{K})$ denotes the set of polynomials with total degree not exceeding $k$ on $\hat{K}$. The set of degrees of freedom $\hat{\Sigma}$ is defined as follows:  
\begin{itemize}
  \item Function value $\hat{v}(\bm{\hat{a}})$ and first order derivatives
  $\hat{\partial}_i \hat{v}(\bm{\hat{a}})$, $i=1,2$ at each vertex;
  \item Function values at $(k - 3)$ nodes on the interior of each edge $\hat{e}$;
  \item Function values at $\frac{(k-1)(k-2)}{2}$ nodes in the interior of $\hat{K}$.
\end{itemize}
It is simple to check that the degrees of freedom given above form a unisolvent set \cite{brenner2007mathematical}.
\paragraph{\underline{\textit{Curved Hermite element}}}
The shape function space for a curved triangle $K$ is given as
$$ P_{K} := \{ v = \hat{v} \circ F_{K}^{-1}, \hat{v} \in \mathcal{P}_{k}(\hat{K})\}, \quad k \geq 3. $$
The set of degrees of freedom $\Sigma_{K}$ is defined as follows: 
\begin{itemize}
  \item Function value $v(\bm{a})$ and 
        first order derivatives $\partial_i v(\bm{a})$, $i=1,2$ at each vertex;
  \item Function values at $(k - 3)$ nodes on the interior of each edge $e$;
  \item Function values at $\frac{(k-1)(k-2)}{2}$ nodes in the interior of $K$.
\end{itemize}
In $\Sigma_{K}$, 
the choice of the nodes on the edge or in the triangle are induced by the nodes in $\hat{\Sigma}$ through $F_K$. 

We now verify the unisolvent property. 
First, the dimension of the shape function space is $\dim P_{K} = \dim \mathcal{P}_{k}(\hat{K}) = \frac{(k+1)(k+2)}{2}$, which coincides with the number of the degrees of freedom.
Therefore, it suffices to show that $v \in P_{K}$ vanishes if it vanishes at all the degrees of freedom in $\Sigma_{K}$. By the definition of $P_{K}$, we know that $v = \hat{v}\circ F_{K}^{-1}$ for some $\hat{v} \in \mP_{k}(\hat{K})$. 
Note that $v$ and its first derivatives vanish at the three vertices of $K$, by the chain rule, we immediately know that $\hat{v}$ and its first derivatives also vanish at the three vertices of $\hat{K}$.
Next, the vanishing of the remaining degrees of freedom in $\Sigma_{K}$ ensures that $\hat{v}$ has $\frac{(k-1)(k-2)}{2}$ zero points inside $\hat{K}$ and $(k - 3)$ zero points inside each edge $\hat{e}$.
Following the unisolvent property of $(\hat{K}, \mP_{k}(\hat{K}), \hat{\Sigma})$, $\hat{v}$ vanishes. Therefore, $v = \hat{v} \circ F_{K}^{-1} \equiv 0$. 

The degrees of freedom ensure that the global finite element space has the $C^{0}$-continuity in $\Omega$ and $C^{1}$-continuity on the vertices for every triangulation $\T_{h}$, i.e.,  
$$ V_{h} = \{ v \in H^{1}(\Omega): v|_{K} \in P_{K}, \forall K \in \mathcal{T}_{h}, \text{ $v$ is $C^{1}$ at all vertices} \}. $$
Enforcing the boundary conditions, we define 
$$ V^{\ell}_{h,0} := \{ v_{h} \in V_{h}, \int_{\Omega} v_{h} \mathrm{d}x = 0 ,\nabla v_{h}(x_{i}) \cdot \ell \text{ is constant } \forall x_{i} \in \mathcal{N}_{h}^{\partial} \}. $$
\begin{remark}[regularity of $\T_{h}$]
  To get the optimal approximation property, we need to assume that $\{\T_{h}\}_{h>0}$ is regular of order $k$, which is consistent with the degree of the polynomial involved in $V_{h}$.
\end{remark}

The following lemmas deal with the scaling argument and the trace estimate. We refer to \cite{bernardi1989optimal} for the proof. 
\begin{lemma}[scaling on curved triangles, see Lemma 2.3 in \cite{bernardi1989optimal}]
  \label{lm:scaling}Let $\{\mathcal{T}_{h}\}_{h>0}$ be regular of order $k$. Let $s$ be an integer, $0 \leq s \leq k+1$, and $p \in [1, \infty]$. For any $K \in \mathcal{T}_{h}$, a function $v$ belongs to $W^{s}_{p}(K)$ if and only if the function $\hat{v} = v \circ F_{K}$ belongs to $W^{s}_{p}(\hat{K})$, and we have
\begin{subequations} \label{eq:scaling} 
\begin{align}
|v|_{W^{s}_{p}(K)} &\leq C_{\rm{S}} 
|\det B_{K}|^{1/p}\|B^{-1}_{K}\|^{s} 
\left(\sum_{i=\min(s,1)}^{s}\|B_{K}\|^{(s - i)}|\hat{v}|_{W^{i}_{p}(\hat{K})} \right), \label{eq:scaling-a} \\
|\hat{v}|_{W^{s}_{p}(\hat{K})} &\leq C_{\rm{S}} 
|\det B_{K}|^{-1/p}\|B_{K}\|^{s} 
\left(\sum_{i=\min(s,1)}^{s} |v|_{W^{i}_{p}(K)} \right), \label{eq:scaling-b}
\end{align}
\end{subequations} 
where $C_{\rm{S}}$ depends continuously on $c_{1}, c_{2}, \cdots, c_{k+1}$ given in Definition \ref{def:triangulation-m}.
\end{lemma}  

By the shape regularity of $\{\mathcal{T}_{h}\}_{h>0}$, 
we know $\|B_{K}\| \eqsim h_{K}$ and $|\det B_{K}| \eqsim h_{K}^{2}$. Thus we obtain
$$ |v|_{W^{s}_{p}(K)}\lesssim h_{K}^{2/p-s}
 \sum_{i=\min(s,1)}^{s} h_{K}^{(s-i)}|\hat{v}|_{W^{i}_{p}(\hat{K})} 
 \text{ and } 
 |\hat{v}|_{W^{s}_{p}(\hat{K})}\lesssim h_{K}^{s-2/p}
 \sum_{i=\min(s,1)}^{s} |v|_{W^{i}_{p}(K)}, $$
with hidden constants depend on the shape-regular constant $\sigma$ and $c_{1}, c_{2}, \cdots, c_{k+1}$.

\begin{lemma}[trace estimate, see Lemma 2.4 in \cite{bernardi1989optimal}]
  \label{lm:trace}Let $\{ \mathcal{T}_{h} \}_{h>0}$ be shape-regular. For any $K \in \mathcal{T}_{h}$ we have 
\begin{equation}\label{eq:trace} 
\|v\|^{2}_{L^{2}(\p K)} \leq C_{\rm{Tr}} (h_{K}^{-1}\|v\|^{2}_{L^{2}(K)} + h_{K}|v|_{H^{1}(K)}^{2}) \qquad \forall v \in H^{1}(K), 
\end{equation} 
where $C_{\rm{Tr}}$ dependents on $c_{1}$ and the shape-regular constant $\sigma$.
\end{lemma} 
\begin{lemma}[inverse estimate]
  \label{lm:inverse}Let $\{\T_{h}\}_{h>0}$ is regular of order $k$. For any $K \in \T_{h}$, the following inverse estimate holds with $0 \leq m < s \leq k + 1$ 
\begin{equation}\label{eq:inverse} 
|v|_{H^{s}(K)}\leq C_{{\rm{Inv}} }h_{K}^{m-s} \sum_{i=\min(1,m)}^{m} |v|_{H^{i}(K)}, 
\qquad \forall v \in P_{K}, 
\end{equation}   
where $C_{\rm{Inv}}$ depends on the polynomial degree $k$, the shape-regular constant $\sigma$ and $c_{1}, c_{2}, \cdots, c_{k+1}$.
\end{lemma} 
\begin{proof}
  If $m\geq 1$, by Lemma \ref{lm:scaling} (scaling on curved triangles) and the norm equivalence on finite dimensional space, we have\
  \begin{equation}\label{eq:inverse-1} 
    \begin{aligned}
      |v|_{H^{s}(K)}&\lesssim h_{K}^{1-s}
     \sum^{s}_{i=1}|\hat{v}|_{H^{i}(\hat{K})}h_{K}^{(s-i)}
      \lesssim h_{K}^{1-s} \sum^{m}_{i=1}|\hat{v}|_{H^{i}(\hat{K})}h_{K}^{(m-i)}.    
    \end{aligned}
  \end{equation} 
  Next for each $1\leq i\leq m$, by Lemma \ref{eq:scaling} (scaling on curved triangles) again, we have $|\hat{v}|_{H^{i}(\hat{K})} \lesssim h_{K}^{i-1} \sum_{j=1}^{i}|v|_{H^{j}(K)}$. Substitute this into \eqref{eq:inverse-1}, we obtained \eqref{eq:inverse}. 
  For the case of $m = 0$, similar arguments lead to 
  $$ \begin{aligned}
    |v|_{H^{s}(K)}&\lesssim 
    h_{K}^{1-s}\sum^{s}_{i=1}|\hat{v}|_{H^{i}(\hat{K})}h_{K}^{(s-i)}
    \lesssim h_{K}^{1-s}\|\hat{v}\|_{L^{2}(\hat{K})}\lesssim h_{K}^{-s}\|v\|_{L^{2}(K)}. 
  \end{aligned} $$
  This completes the proof. 
\end{proof}

As one of the main focuses of our work, the following theorem gives the approximation property with oblique boundary conditions. The proof is a nontrivial, and  is threfore postponed to Appendix \ref{app:approximation} for the reader's convenience. 
\begin{theorem}[approximation property of $V^{\ell}_{h,0}$]
  \label{thm:app_Vh_ell}Assume that $\{ \mathcal{T}_{h} \}_{h>0}$ is regular of order $k$, 
  $u \in H^{s}(\Omega) \cap H^{2}_{\ell,0}(\Omega)$, and the oblique vector field $\ell$ is piecewise $C^{s-1}$ with $2\leq s \leq k+1$. 
  Then there exists a quasi-interpolation $\Pi^{\ell}_{h}: H^{2}_{\ell,0}(\Omega) \rightarrow V^{\ell}_{h,0}$ satisfying  
  \begin{subequations}
  \begin{align}
  |u - \Pi_{h}^{\ell}u|_{H^{m}(K)} \lesssim h_{K}^{s-m}\|u\|_{H^{s}(\omega_{K})}, \qquad m=1,2,  \label{eq:app-Pi-ell-K}  \\
  \|\nabla (u - \Pi_{h}^{\ell}u)\|_{L^{2}(\p K)}\lesssim h_{K}^{s-1-1/2}\|u\|_{H^{s}(\omega_{K})}. \label{eq:app-Pi-ell-F} 
  \end{align} 
  \end{subequations}
  Here, $\omega_{K}$ is the local neighborhood of $K \in \T_{h}$ and the hidden constants depend on $\|\ell\|_{W^{s-1}_{\infty}(\p \omega_{K} \cap \p\Omega)}$ , the polynomial degree $k$, the shape-regular constant $\sigma$ and $c_{1}, c_{2}, \cdots, c_{k+1}$.
\end{theorem}

\subsection{The discrete Miranda-Talenti-type identity for oblique boundary conditions}
In this subsection, we establish the Miranda-Talenti estimate for oblique boundary conditions in the discrete level by using the $C^{1}$-continuity on the vertices of the finite element space.

Define the jump of the normal derivative on an interior edge $F=\p K^{+} \cap \p K^{-}$ as 
$$ [\pd{v}{n}{}] := \frac{\p v^{+}}{\p n^{+}} + \frac{\p v^{-}}{\p n^{-}}  \quad \text{ where } \quad v^{\pm} := v|_{K^{\pm}}, $$
where $n^{\pm}$ is the unit outward normal vector of $K^{\pm}$.
For any interior edge $F \in \F^{i}_{h}$, we specify a tangent direction denoted as $t_{F}$. Let $\bo{x}^{(1)}_{F}$ and $\bo{x}^{(2)}_{F}$ be the two vertices of $F$ (along the direction $t_{F}$). 
Recall $\ell = (\ell_{1},\ell_{2})$ is the oblique vector field defined on $\p\Omega$, and we denote 
$$ \frac{\partial v}{\partial \ell} := \ell_{1}\partial_{1} v + \ell_{2}\partial_{2} v
\quad\text{ and }\quad    
  \frac{\partial v}{\partial \ell^{\perp}} := -\ell_{2}\partial_{1} v + \ell_{1}\partial_{2} v $$
as the directional derivatives along $\ell$ and its perpendicular direction $\ell^{\perp} = (-\ell_{2}, \ell_{1})$. Moreover, we denote
  $ \dot{\left(\frac{\partial v}{\partial \ell}\right)}:= \frac{ \mathrm{d}}{\dphi}\left(\frac{\partial v}{\partial \ell}(\bo{x}(\varphi))\right). $
  
\begin{theorem}[discrete Miranda-Talenti-type identity for oblique boundary conditions]
  \label{thm:dis-mt-obl}Let $\Omega$ be a bounded domain with $C^{2}$ boundary. 
  The oblique vector field $\ell$ is of class $C^{1}$. 
  Let $\mathcal{T}_{h}$ be an exact conforming triangulation of $\Omega$. 
  For any $v_{h} \in V_{h}$, it holds that 
  \begin{equation}\label{eq:dis-MT-obl} 
  \begin{aligned}
  \sum_{K\in\T_{h}}\int_{K}(\Delta v_{h})^{2}\dx=
  \sum_{K\in\T_{h}}&\int_{K}|D^{2}v_{h}|^{2}\dx + 
  2\sum_{F\in\F^{i}_{h}}\int_{F}[\frac{\p v_{h}}{\p n}]\frac{\p^{2} v_{h}}{\p t_{F}^{2}}\ds \\ 
  &+\sum_{F\in\F^{\p}_{h}}\int_{F}
  |\nabla v_{h}|^{2}(\dot{\theta} - \chi) - 
  2\frac{\p v_{h}}{\p\ell^{\perp}}\dot{\left(\frac{\p v_{h}}{\p \ell}\right)}  \ds. 
  \end{aligned}  
  \end{equation} 
\end{theorem} 
\begin{proof}

 \underline{\textit{Step 1:}}
 For a straight triangle $K \in \mathcal{T}_{h}$, the normal vector of $\partial K$ is piecewise constant. 
 Using integration by part, we obtain (see \cite[Equ. (3.7)]{smears2013discontinuous})
 \begin{equation}\label{eq:dis-mt-str-1} 
  \begin{aligned}
    \int_{K}(\Delta v_{h})^{2}\dx
    &=\int_{K}|D^{2}v_{h}|^{2}\dx + \int_{\partial K}
    \frac{\partial v_{h}}{\partial n}\Delta v_{h} - 
    \nabla \frac{\partial v_{h}}{\partial n} \cdot \nabla v_{h} \ds \\ 
    &=\int_{K}|D^{2}v_{h}|^{2}\dx + \int_{\partial K}
    \frac{\partial v_{h}}{\partial n} \frac{\partial^{2} v_{h}}{\partial t^{2}}  - 
    \frac{\partial^{2} v_{h}}{\partial t\partial n}\frac{\partial v_{h}}{\partial t} \ds.
   \end{aligned}
 \end{equation} 
 Here, the common term 
 $\int_{\partial K}\frac{\partial^{2} v_{h}}{\partial n^2}\frac{\partial v_{h}}{\partial n}\ds $ is cancelled  in the last step. Using integration by part on the edge $F \subset \partial K$, we have 
\begin{equation*}
  \begin{aligned}
    \int_{\partial K}
  \frac{\partial^{2} v_{h}}{\partial t\partial n}\frac{\partial v_{h}}{\partial t}\ds 
  &= \sum_{F \subset \partial K}\int_{F}
  \frac{\p^{2} v_{h}}{\p t_{F}\p n} \frac{\p v_{h}}{\p t_{F}}\ds \\ 
  &= \sum_{F \subset \partial K}\left(-\int_{F}
  \frac{\p v_{h}}{\p n}\frac{\p^{2} v_{h}}{\p t_{F}^{2}}\ds + 
  \frac{\p v_{h}}{\p n}\frac{\p v_{h}}{\p t_{F}}\Bigg|^{\bo{x}_{F}^{(2)}}_{\bo{x}_{F}^{(1)}}\right).
  \end{aligned} 
\end{equation*} 
 Hence, \eqref{eq:dis-mt-str-1} can be reformulated as
\begin{equation}\label{eq:dis-mt-str} 
  \int_{K}(\Delta v_{h})^{2}\dx
  =\int_{K}|D^{2}v_{h}|^{2}\dx + 
  \sum_{F \subset \p K}\left(
  2\int_{F}\frac{\p v_{h}}{\p n}\frac{\p^{2} v_{h}}{\p t_{F}^{2}}\ds - 
  \frac{\p v_{h}}{\p n}\frac{\p v_{h}}{\p t_{F}}\Bigg|^{\bo{x}_{F}^{(2)}}_{\bo{x}_{F}^{(1)}}\right).
\end{equation} 
\underline{\textit{Step 2:}}
A direct calculation shows that 
$$ (\Delta v_{h})^{2}-|D^{2}v_{h}|^{2} = 
\frac{\p}{\p x_{1}}
\left(\p_{1}v_{h}\p_{22}v_{h} - \p_{2}v_{h}\p_{12}v_{h}\right) - 
\frac{\p}{\p x_{2}}
\left(\p_{1}v_{h}\p_{12}v_{h} - \p_{2}v_{h}\p_{11}v_{h}\right). $$
For any curved triangle $K \in \T_{h}$, apply the divergence Theorem 
$$ \begin{aligned}
  & \int_{K} (\Delta v_{h})^{2}-|D^{2}v_{h}|^{2} \dx \\
&= \int_{\p K} 
n_{1}\left(\p_{1}v_{h}\p_{22}v_{h} - \p_{2}v_{h}\p_{12}v_{h}\right) - 
n_{2}\left(\p_{1}v_{h}\p_{12}v_{h} - \p_{2}v_{h}\p_{11}v_{h}\right) \ds \\ 
&= \sum^{3}_{i = 1,F_{i}\subset\p K}\underbrace{\int_{F_{i}} 
n_{1}\left(\p_{1}v_{h}\p_{22}v_{h} - \p_{2}v_{h}\p_{12}v_{h}\right) - 
n_{2}\left(\p_{1}v_{h}\p_{12}v_{h} - \p_{2}v_{h}\p_{11}v_{h}\right) \ds}_{I_{i}}.
\end{aligned}$$
In the last step, we denote $\p K = \cup^{3}_{i=1}F_{i}$ where $F_{1}$ and $F_{2}$ are the two straight edges and $F_{3} \subset \p\Omega $ is the curved edge. On the straight edge (i.e. $i=1,2$), the same calculation as in \textit{Step 1} leads to 
\begin{equation}\label{eq:dis-mt-Ii} 
  \begin{aligned}
    I_{i} &
    = 2\int_{F_{i}}\frac{\p v_{h}}{\p n}\frac{\p^{2} v_{h}}{\p t_{F_{i}}^{2}}\ds 
     - \left(\frac{\p v_{h}}{\p n}\frac{\p v_{h}}{\p t_{F_{i}}}\right)
     \Bigg|^{\bo{x}_{F_{i}}^{(2)}}_{\bo{x}_{F_{i}}^{(1)}}, \quad i = 1,2.    
  \end{aligned} 
\end{equation}    
On the curved edge $F_{3}$, we first use the directional derivatives $\frac{\p v_{h}}{\p \ell}$ and $\frac{\p v_{h}}{\p \ell^{\perp}} $ to represent $\p_{1}v_{h}$ and $\p_{2}v_{h}$, i.e., 
 $$ \begin{aligned}
  \p_{1}v_{h} &= \frac{\p v_{h}}{\p \ell}\ell_{1} - \frac{\p v_{h}}{\p \ell^{\perp}}\ell_{2}
  \quad \text{ and } \quad
  \p_{2}v_{h} &= \frac{\p v_{h}}{\p \ell}\ell_{2} + \frac{\p v_{h}}{\p \ell^{\perp}}\ell_{1}.
 \end{aligned} $$
 Then, substituting above formula into $I_{3}$ and parametrizing $F_{3} \subset \p\Omega$ by the arc length $\varphi$ through the $C^{2}$ mapping, $\bo{x}: [L_{1},L_{2}] \rightarrow \mathbb{R}^{2}$, we obtain
 \begin{equation}\label{eq:dis-mt-I3-1} 
  \begin{aligned}
    I_{3} &= \int_{F_{3}} 
    n_{1}\left(\p_{1}v_{h}\p_{22}v_{h} - \p_{2}v_{h}\p_{12}v_{h}\right) - 
    n_{2}\left(\p_{1}v_{h}\p_{12}v_{h} - \p_{2}v_{h}\p_{11}v_{h}\right) \ds\\ 
    &= \int^{L_{2}}_{L_{1}} \frac{\p v_{h}}{\p \ell}
    \left(\ell_{2}\p_{11}v_{h}n_{2}-\ell_{2}\p_{12}v_{h}n_{1}
         -\ell_{1}\p_{21}v_{h}n_{2}+\ell_{1}\p_{22}v_{h}n_{1}\right) \\ 
    &\quad + \frac{\p v_{h}}{\p \ell^{\perp}}
    \left(\ell_{1}\p_{11}v_{h}n_{2}-\ell_{1}\p_{12}v_{h}n_{1}
         +\ell_{2}\p_{21}v_{h}n_{2}-\ell_{2}\p_{22}v_{h}n_{1}\right)  \dphi.
  \end{aligned}
 \end{equation} 
Taking the derivative of $\frac{\p v_{h}}{\p \ell}$ and $\frac{\p v_{h}}{\p \ell^{\perp}}$ with respect to the arc length $\varphi$, and noting that $(n_{1},n_{2})=(\dot{\bo{x}}_{2},-\dot{\bo{x}}_{1})$, we obtain
\begin{equation}\label{eq:dis-mt-ell} 
  \begin{aligned}
    \dot{\left(\frac{\p v_{h}}{\p \ell}\right)} 
    &= \left(
    \ell_{1}\p_{11}v_{h}\dot{\bo{x}}_{1}+\ell_{1}\p_{12}v_{h}\dot{\bo{x}}_{2}+
    \ell_{2}\p_{21}v_{h}\dot{\bo{x}}_{1}+\ell_{2}\p_{22}v_{h}\dot{\bo{x}}_{2}\right) \\
    &\quad + \dot{\ell_{1}}\p_{1}v_{h} + \dot{\ell_{2}}\p_{2}v_{h} \\
    &= \left(
    -\ell_{1}\p_{11}v_{h}n_{2}+\ell_{1}\p_{12}v_{h}n_{1}
    -\ell_{2}\p_{21}v_{h}n_{2}+\ell_{2}\p_{22}v_{h}n_{1}\right) \\
    &\quad + \dot{\ell_{1}}\p_{1}v_{h} + \dot{\ell_{2}}\p_{2}v_{h}, \\
  \end{aligned} 
\end{equation} 
 and similarly
 \begin{equation}\label{eq:dis-mt-ellperp} 
  \begin{aligned}
    \dot{\left(\frac{\p v_{h}}{\p \ell^{\perp}}\right)} 
    &= \left(
      \ell_{2}\p_{11}v_{h}n_{2}-\ell_{2}\p_{12}v_{h}n_{1}
      -\ell_{1}\p_{21}v_{h}n_{2}+\ell_{1}\p_{22}v_{h}n_{1}\right) \\
    &\quad - \dot{\ell_{2}}\p_{1}v_{h} + \dot{\ell_{1}}\p_{2}v_{h}. \\
   \end{aligned}
 \end{equation} 
Substituting \eqref{eq:dis-mt-ell} and \eqref{eq:dis-mt-ellperp} into \eqref{eq:dis-mt-I3-1}, we arrive at
$$ 
\begin{aligned}
  I_{3} &= \int^{L_{2}}_{L_{1}} 
       \left((\frac{\p v_{h}}{\p \ell})^{2}+(\frac{\p v_{h}}{\p \ell^{\perp}})^{2}\right)
      (\ell_{1}\dot{\ell}_{2}-\ell_{2}\dot{\ell}_{1}) +
      \frac{\p v_{h}}{\p \ell}\dot{\left(\frac{\p v_{h}}{\p \ell^{\perp}}\right)}-
      \frac{\p v_{h}}{\p \ell^{\perp}}\dot{\left(\frac{\p v_{h}}{\p \ell}\right)}
      \dphi.
\end{aligned}
$$
Using integration by part, 
$|\nabla v_{h}|^{2}=(\frac{\p v_{h}}{\p \ell})^{2}+(\frac{\p v_{h}}{\p \ell^{\perp}})^{2}$ and 
the identity $(\ell_{1}\dot{\ell}_{2}-\ell_{2}\dot{\ell}_{1})=\dot{\theta}-\chi$ (see \cite[Page. 48]{maugeri2000elliptic}), we have 
\begin{equation}\label{eq:dis-mt-I3} 
  \begin{aligned}
    I_{3} &= \int^{L_{2}}_{L_{1}} 
        |\nabla v_{h}|^{2}(\dot{\theta}-\chi)
        -2\frac{\p v_{h}}{\p \ell^{\perp}}\dot{\left(\frac{\p v_{h}}{\p \ell}\right)}
        \dphi + 
        \left(\frac{\p v_{h}}{\p \ell}\frac{\p v_{h}}{\p \ell^{\perp}}\right)
        \Bigg|^{\bo{x}^{(2)}_{F_{3}}}_{\bo{x}^{(1)}_{F_{3}}}.
  \end{aligned}
\end{equation} 
Finally, combing \eqref{eq:dis-mt-Ii} and \eqref{eq:dis-mt-I3}, we have, on the boundary curved element $K$, 
\begin{equation}\label{eq:dis-mt-cur} 
  \begin{aligned}
    \int_{K}(\Delta v_{h})^{2}\dx
  &=\int_{K}|D^{2}v_{h}|^{2}\dx + 
  \sum_{i=1}^{2}\left(
    2\int_{F_{i}}\frac{\p v_{h}}{\p n}\frac{\p^{2} v_{h}}{\p t_{F_{i}}^{2}}\ds 
    - \left(\frac{\p v_{h}}{\p n}\frac{\p v_{h}}{\p t_{F_{i}}}\right)
    \Bigg|^{\bo{x}_{F_{i}}^{(2)}}_{\bo{x}_{F_{i}}^{(1)}}  
  \right)  \\
  &\quad +\int_{F_{3}} 
  |\nabla v_{h}|^{2}(\dot{\theta}-\chi)
  -2\frac{\p v_{h}}{\p \ell^{\perp}}\dot{\left(\frac{\p v_{h}}{\p \ell}\right)}
  \ds + 
  \left(\frac{\p v_{h}}{\p \ell}\frac{\p v_{h}}{\p \ell^{\perp}}\right)
  \Bigg|^{\bo{x}^{(2)}_{F_{3}}}_{\bo{x}^{(1)}_{F_{3}}}.
  \end{aligned}
\end{equation} 
\underline{\textit{Step 3:}}
Using \eqref{eq:dis-mt-str} and \eqref{eq:dis-mt-cur}, we now sum over all the triangles $K \in \mathcal{T}_{h}$ to obtain
$$ \begin{aligned}
\sum_{K\in\T_{h}}\int_{K}&(\Delta v_{h})^{2}\dx=\sum_{K\in\T_{h}}\int_{K}|D^{2}v_{h}|^{2}\dx\\ 
+&2\sum_{F\in\F^{i}_{h}}\int_{F}[\frac{\p v_{h}}{\p n}]\frac{\p^{2} v_{h}}{\p t_{F}^{2}}\ds 
-\underbrace{\sum_{F\in\F^{i}_{h}}\left(
  \frac{\p v_{h}^{+}}{\p n^{+}}\frac{\p v_{h}^{+}}{\p t_{F}}+
  \frac{\p v_{h}^{-}}{\p n^{-}}\frac{\p v_{h}^{-}}{\p t_{F}}\right)
  \Bigg|^{\bo{x}^{(2)}_{F}}_{\bo{x}^{(1)}_{F}}}_{J_{1}}  \\ 
+&\sum_{F\in\F^{\p}_{h}}
\int_{F}|\nabla v_{h}|^{2}(\dot{\theta} - \chi) - 
2\frac{\p v_{h}}{\p\ell^{\perp}}\dot{\left(\frac{\p v_{h}}{\p \ell}\right)}  \ds+\underbrace{\sum_{F\in\F^{\p}_{h}}\left(
  \frac{\p v_{h}}{\p \ell}\frac{\p v_{h}}{\p \ell^{\perp}}\right)
  \Bigg|^{\bo{x}^{(2)}_{F}}_{\bo{x}^{(1)}_{F}}}_{J_{2}}.
\end{aligned} $$
For any interior edge $F=K^{+}\cap K^{-}$, the $C^{1}$-continuity on the vertices implies that 
$$ \frac{\p v_{h}^{+}}{\p t_{F}}(\bo{x})=\frac{\p v_{h}^{-}}{\p t_{F}}(\bo{x}), \qquad 
   \frac{\p v_{h}^{+}}{\p n^{+}}(\bo{x})=-\frac{\p v_{h}^{-}}{\p n^{-}}(\bo{x}),$$
here $\bo{x} \in \N_{h}$ is the vertex of $F$. This shows that $J_{1}=0.$
Let $\bo{x}=F^{+}\cap F^{-}$ be a boundary vertex with $F^{+}, F^{-}\in\F^{\p}_{h}$, we have 
$$ J_{2} = \sum_{\bo{x}\in \N^{\p}_{h}}
\left((\frac{\p v_{h}}{\p \ell}\frac{\p v_{h}}{\p \ell^{\perp}})\Big|_{F^{+}} - 
      (\frac{\p v_{h}}{\p \ell}\frac{\p v_{h}}{\p \ell^{\perp}})\Big|_{F^{-}}
\right)(\bo{x}).
$$
Similarly, the $C^{1}$-continuity on the vertices leads to $J_{2} = 0$, which concludes the proof. 
\end{proof}

\section{Numerical scheme and analysis}\label{sec:analysis}
In this section, we will give the numerical scheme for approximating \eqref{eq:obl-der-problem}. We first define a broken norm as
$$ \|v\|_{h}^{2}:=\sum_{K\in\T_{h}}|v|^{2}_{H^{2}(K)} + 
   \chi_{0}\sum_{F\in\F^{\p}_{h}}\int_{F}|\nabla v|^{2}\ds, $$
where we recall $\chi_{0} = \displaystyle\min_{\p\Omega}(\dot{\theta} - \chi)$. For any $v \in V_{h,0}^{\ell}$, $\|v\|_{h} = 0$ implies that $D^{2}v |_{K} \equiv 0 $ on each triangle $K\in\T_{h}$. Due to $C^{1}$-continuity on the vertices, we immediately know that $v$ is a linear polynomial on $\Omega$. Since 
$ \sum_{F\in\F^{\p}_{h}}\int_{F}|\nabla v|^{2}\ds = 0 $ and $
\int_{\Omega}v\dx=0 $,
we conclude that $v \equiv 0 $.
Hence, $\|\cdot\|_{h}$ is a norm on $V_{h,0}^{\ell}$.

In light of \eqref{eq:bilinear-oblique}, we define the bilinear form $b_{h}:V_{h}\times V_{h}\rightarrow\mathbb{R}$ as 
\begin{equation}\label{eq:bilinear-dis} 
    b_{h}(w_{h},v_{h}) :=
    \sum_{K\in\T_{h}}(\gamma A:D^{2}w_{h},\Delta v_{h})_{K} + 
    \frac{2 - \sqrt{1-\varepsilon }}{2} s_{h}(w_{h},v_{h}),
\end{equation}  
where the stabilized term $s_{h}: V_{h} \times V_{h} \rightarrow \mathbb{R}$ is defined by 
$$ s_{h}(w_{h}, v_{h}):= 
-2\sum_{F\in\F^{i}_{h}}\int_{F}
[\frac{\p w_{h}}{\p n}]\frac{\p^{2} v_{h}}{\p t_{F}^{2}}\ds
+2\sum_{F\in\F^{\p}_{h}}\int_{F}
\dot{\left(\frac{\p w_{h}}{\p \ell}\right)}\frac{\p v_{h}}{\p\ell^{\perp}}\ds,$$
which is naturally induced by the jump and boundary terms in the discrete Miranda-Talenti-type identity \eqref{eq:dis-MT-obl}. 
We are now ready to define the numerical scheme as follows: Find $u_{h} \in V_{h,0}^{\ell}$ such that 
\begin{equation}\label{eq:num-scheme} 
    b_{h}(u_{h},v_{h}) = l_{h}(v_{h})
   \qquad \forall v_{h} \in V_{h,0}^{\ell}, 
\end{equation} 
where the linear form $l_{h}: V_{h} \rightarrow \mathbb{R}$ is defined by $l_{h}(v_{h}) := \sum_{K\in\mathcal{T}_{h}}(\gamma f, \Delta v_{h})_{K}$.

\begin{lemma}[consistence]
    \label{lm:consistence}Let $u \in H^{2}_{\ell,0}(\Omega)$ be the $H^{2}$ strong solution of \eqref{eq:obl-der-problem}, then $u$ satisfies the numerical scheme \eqref{eq:num-scheme}, i.e., $b_{h}(u,v_{h}) = l_{h}(v_{h})$ for any $v_{h}\in V_{h,0}^{\ell}$.
\end{lemma} 
\begin{proof}
    $u \in H^{2}_{\ell,0}(\Omega)$ leads to 
    $$ [\frac{\p u}{\p n}]\Big|_{F} \equiv 0,\quad \forall F\in\F^{i}_{h} 
    \quad \text{ and } \quad 
    \dot{\left(\frac{\p u}{\p \ell}\right)}\Big|_{F} \equiv 0,\quad \forall F\in\F^{\p}_{h}. $$
    Thus we know $s_{h}(u,v_{h})=0$. Since $u$ is the $H^{2}$ strong solution of \eqref{eq:obl-der-problem}, we have  
     $$b_{h}(u,v_{h}) = \sum_{K\in\T_{h}}(\gamma A:D^{2}u,\Delta v_{h})_{K} 
      = \sum_{K\in\T_{h}}(\gamma f,\Delta v_{h})_{K} =l_{h}(v_{h}).$$
    That conclude the proof.  
\end{proof}
The following coercivity result follows directly from Theorem \ref{thm:dis-mt-obl} (discrete Miranda-Talenti-type identity for oblique boundary conditions).   
\begin{lemma}[coercivity]
    \label{lm:coer}Under the Assumption \ref{ass:problem}, let $\mathcal{T}_{h}$ be a exact conforming triangulation of $\Omega$, then 
\begin{equation}\label{eq:coer} 
b_{h}(v_{h},v_{h}) \geq (1-\sqrt{1-\varepsilon})\|v_{h}\|^{2}_{h} \qquad v_{h}\in V_{h}.
\end{equation}
Here, recall $\varepsilon$ is the parameter in the Cordes condition \eqref{def:cordes}.  
\end{lemma} 
\begin{proof}
 Combining Lemma \ref{lm:cordes} (property of Cordes condition) and the Cauchy-Schwarz inequality, we obtain  
 $$ \begin{aligned}
 b_{h}(v_{h},v_{h})
 &=\sum_{K\in\T_{h}}\|\Delta v_{h}\|_{L^{2}(K)}^{2} +
 \sum_{K\in\T_{h}}((\gamma A - I):D^{2}v_{h},\Delta v_{h})_{K}\\
 &\quad + \frac{2-\sqrt{1-\varepsilon}}{2}s_{h}(v_{h},v_{h}) \\
 &\geq\sum_{K\in\T_{h}}\|\Delta v_{h}\|_{L^{2}(K)}^{2} -\sqrt{1-\varepsilon}\sum_{K\in\T_{h}}
 \|D^{2}v_{h}\|_{L^{2}(K)}\|\Delta v_{h}\|_{L^{2}(K)}\\ 
 &\quad + \frac{2-\sqrt{1-\varepsilon}}{2}s_{h}(v_{h},v_{h})
 ~~~~~~~~~~~~~~~~~~~~(\text{by Lemma \ref{lm:cordes}})\\ 
 &\geq\frac{2-\sqrt{1-\varepsilon}}{2}
 \left(\sum_{K\in\T_{h}}\|\Delta v_{h}\|_{L^{2}(K)}^{2} + s_{h}(v_{h},v_{h})\right)\\ 
 &\quad -\frac{\sqrt{1-\varepsilon}}{2}\sum_{K\in\T_{h}}\|D^{2}v_{h}\|^{2}_{L^{2}(K)}.
 ~~~~~~~~~~~~(\text{by Cauchy-Schwarz})\\
 \end{aligned} $$
 Applying \eqref{eq:dis-MT-obl} in Theorem \ref{thm:dis-mt-obl} (discrete Miranda-Talenti-type identity for oblique boundary conditions), we have 
 $$ \begin{aligned}
    b_{h}(v_{h},v_{h}) 
    &\geq (1-\sqrt{1-\varepsilon})\sum_{K\in\T_{h}}\|D^{2}v_{h}\|_{L^{2}(K)}^{2} \\ 
    &+ \frac{2-\sqrt{1-\varepsilon}}{2}\chi_{0}
    \sum_{F\in\F^{\p}_{h}}\int_{F}|\nabla v|^{2}\ds \geq (1-\sqrt{1-\varepsilon})\|v_{h}\|^{2}_{h}.
 \end{aligned}$$
 This concludes the proof.
\end{proof}

We note that the coercivity constant (namely, $1 - \sqrt{1 - \varepsilon}$) under the broken norm $\|\cdot\|_{h}$ is exactly the same as that for the PDE theory. 
\begin{remark}[approximation of $\varepsilon$ in Cordes condition]
\label{rmk:app-eps}
We note that the bilinear form \eqref{eq:bilinear-dis} explicitly uses the constant $\varepsilon$ in the Cordes condition \eqref{eq:cordes}. 
As discussed in \cite{wu2021c0}, if the optimal value of $\varepsilon$ is not easy to compute, a simple modification of \eqref{eq:bilinear-dis} reads 
$$ \tilde{b}_{h}(w_{h},v_{h}) :=
\sum_{K\in\T_{h}}(\gamma A:D^{2}w_{h},\Delta v_{h})_{K} + 
\frac{2 - \sqrt{1-\tilde{\varepsilon}}}{2} s_{h}(w_{h},v_{h}), $$
where $\tilde{\varepsilon}$ is an approximation of $\varepsilon$ that satisfies 
$ \sqrt{1 - \tilde{\varepsilon}} + \frac{1-\varepsilon}{\sqrt{1-\tilde{\varepsilon}}} < 2.$
As a result, the coercivity constant becomes $1 - \frac{\sqrt{1 - \tilde{\varepsilon}}}{2} - \frac{1-\varepsilon}{2\sqrt{1-\tilde{\varepsilon}}}$. Even if there is no a priori estimate of $\varepsilon$, one may simply take $\tilde{\varepsilon} = 0$, which leads to the coercivity constant $\frac{\varepsilon}{2}$.
We refer to \cite[Section 4.4]{wu2021c0} for more details.
\end{remark}

Next, we introduce a Poincar\'e-type estimate which is helpful in the analysis. The proof is given in Appendix \ref{app:poincare} for the reader's convenience. 
\begin{lemma}[Poincar\'e-type estimate]
    \label{lm:poin}Let $K\in\T_{h}$ be a curved triangle with a curved edge $F$.  Let $u \in H^{1}(K)$, we have 
    $$ 
    \|u\|^{2}_{L^{2}(K)} \leq 2(1+c_{K})^{2} 
    \left(h_{K}\|u\|_{L^{2}(F)}^{2} + h^{2}_{K}|u|^{2}_{H^{1}(K)}\right),
    $$
where $c_{K}$ is defined in \eqref{eq:cK}.    
\end{lemma}

We are now ready to give the boundedness result by combing Lemma \ref{lm:inverse} (inverse estimate), Lemma \ref{lm:trace} (trace estimate), and Lemma \ref{lm:poin} (Poincar\'e-type estimate). 
\begin{lemma}[boundedness]\label{lm:boundedness}Under the Assumption \ref{ass:problem}, let $\{\T_{h}\}_{h>0}$ be regular of order $2$. Then, it holds that 
\begin{equation}\label{eq:boundedness} 
    b_{h}(w_{h},v_{h}) \leq C_{\rm{b}} \|w_{h}\|_{h}\|v_{h}\|_{h} \qquad w_{h},v_{h} \in V^{\ell}_{h,0}. 
\end{equation}
Here the constant $C_{\rm{b}}$ depends on the the polynomial degree $k$, the shape-regular constant $\sigma$ and $c_{1}, c_{2}, c_{3}$.
\end{lemma}
\begin{proof}
\underline{\textit{Step 1}:} For the first term of $b_{h}(w_{h},v_{h})$, by Lemma \ref{lm:cordes} (property of Cordes condition) we have 
\begin{equation*} 
    \begin{aligned}
        I_{1} 
        &:= \left|\sum_{K\in\T_{h}}((\gamma A-I):D^{2}w_{h},\Delta v_{h})_{K} + 
        \sum_{K\in\T_{h}}(\Delta w_{h},\Delta v_{h})_{K}\right| \lesssim \|w_{h}\|_{h}\|v_{h}\|_{h}.\\ 
    \end{aligned}
\end{equation*} 
\underline{\textit{Step 2}:} We estimate the first term of $s_{h}(w_{h},v_{h})$. 
Note that $[\frac{\p w_{h}}{\p n}]$ vanishes on the two vertices of $F \in \F^{i}_{h}$, a Poincar\'e inequality on $F$ leads to 
\begin{equation}\label{eq:bound-Fi-0} 
    \begin{aligned}
        I_{2} &:= 
        \left| \sum_{F\in\F^{i}_{h}}\int_{F}[\frac{\p w_{h}}{\p n}]\frac{\p^{2} v_{h}}{\p t_{F}^{2}}\ds\right| \\ 
        &\leq\left(\sum_{F\in\F^{i}_{h}}h_{F}^{-1}\|[\frac{\p w_{h}}{\p n}]\|_{L^{2}(F)}^{2}\right)^{1/2}
        \left(\sum_{F\in\F^{i}_{h}}h_{F}\|\frac{\p^{2} v_{h}}{\p t_{F}^{2}}\|_{L^{2}(F)}^{2}\right)^{1/2}\\
        &\lesssim\left(\sum_{F\in\F^{i}_{h}}h_{F}\|\frac{\p}{\p t_{F}}[\frac{\p w_{h}}{\p n}]\|_{L^{2}(F)}^{2}\right)^{1/2}
        \left(\sum_{F\in\F^{i}_{h}}h_{F}\|\frac{\p^{2} v_{h}}{\p t_{F}^{2}}\|_{L^{2}(F)}^{2}\right)^{1/2}. 
        \end{aligned} 
\end{equation} 
We define an indicator function $\delta: \T_{h} \rightarrow \{0,1\}$ as 
$ \delta(K) := \begin{cases}
    1, & |\p K\cap\p \Omega|\neq 0\\
    0, & |\p K\cap\p \Omega|= 0 \\ 
\end{cases}.  $
For an interior edge $F=K^{+}\cap K^{-}$, Lemma \ref{lm:trace} (trace estimate) leads to 
\begin{equation*} 
    \begin{aligned}
        h_{F}\|\frac{\p}{\p t_{F}}[\frac{\p w_{h}}{\p n}]\|_{L^{2}(F)}^{2} &\lesssim 
        \sum_{K \in \{K^{+},K^{-}\}}|w_{h}|_{H^{2}(K)}^{2} + h_{K}^2|w_{h}|_{H^{3}(K)}^{2}\\ 
        &\lesssim
        \sum_{K \in \{K^{+},K^{-}\}}|w_{h}|^{2}_{H^{2}(K)} + \delta(K)|w_{h}|^{2}_{H^{1}(K)}.\\
    \end{aligned}
\end{equation*} 
In the last step, we use the standard inverse estimate \cite{brenner2007mathematical} for the term  $|w_{h}|_{H^{3}(K)}^{2}$ if $K$ is a straight triangle. 
If $K$ is a curved triangle,  we apply Lemma \ref{lm:inverse} (inverse estimate), which requires $\{T_{h}\}_{h>0}$ to be regular of order $2$. 
Then, applying Lemma \ref{lm:poin} (Poincar\'e-type estimate), we arrive at
\begin{equation}\label{eq:bound-Fi-1} 
\begin{aligned}
\sum_{F\in\F^{i}_{h}}h_{F}\|\frac{\p}{\p t_{F}}[\frac{\p w_{h}}{\p n}]\|_{L^{2}(F)}^{2} &\lesssim
\sum_{K\in\T_{h}}|w_{h}|^{2}_{H^{2}(K)} + \delta(K)|w_{h}|^{2}_{H^{1}(K)}\\ 
&\lesssim \sum_{K\in\T_{h}}|w_{h}|^{2}_{H^{2}(K)}+ \sum_{F\in\F^{\p}_{h}}\|\nabla w_{h}\|_{L^{2}(F)}^{2} \lesssim \|w_{h}\|^{2}_{h}.
\end{aligned}
\end{equation}  
Similar arguments as \eqref{eq:bound-Fi-1} lead to 
$ \sum_{F\in\F^{i}_{h}}h_{F}\|\frac{\p^{2} v_{h}}{\p t_{F}^{2}}\|_{L^{2}(F)}^{2} \lesssim 
 \|v_{h}\|^{2}_{h}.$
This, combined with \eqref{eq:bound-Fi-1}, leads to $I_{2}\lesssim \|w_{h}\|_{h}\|v_{h}\|_{h}$.

\underline{\textit{Step 3}:} Now we estimate the second term of $s_{h}(w_{h},v_{h})$. By the boundary condition of $V_{h,0}^{\ell}$, we know 
    $\int_{F}\dot{(\frac{\p w_{h}}{\p \ell})}\ds = \frac{\p w_{h}}{\p \ell}(\bo{x}^{(2)}_{F})- \frac{\p w_{h}}{\p \ell}(\bo{x}^{(1)}_{F}) = 0$ for any boundary edge $F$. 
Denote $c_{F} := \fint_{F}\frac{\p v_{h}}{\p\ell^{\perp}}\ds$, a Poincar\'e inequality on $F$ leads to 
\begin{equation}\label{eq:bound-Fp-0} 
    \begin{aligned}
        I_{3} &:= \left|\sum_{F\in\F^{\p}_{h}}\int_{F}
        \dot{\left(\frac{\p w_{h}}{\p \ell}\right)}\frac{\p v_{h}}{\p\ell^{\perp}}\ds\right|
        =\left|\sum_{F\in\F^{\p}_{h}}\int_{F}
        \dot{\left(\frac{\p w_{h}}{\p \ell}\right)}(\frac{\p v_{h}}{\p\ell^{\perp}}-c_{F})\ds\right|\\ 
        & \leq 
        \left(\sum_{F\in\F^{\p}_{h}}h_{F}\|\dot{\left(\frac{\p w_{h}}{\p \ell}\right)}\|_{L^{2}(F)}^{2}  \right)^{1/2}
        \left(\sum_{F\in\F^{\p}_{h}}h_{F}^{-1}\|\frac{\p v_{h}}{\p\ell^{\perp}}-c_{F}\|^{2}_{L^{2}(F)}\right)^{1/2}\\
        & \lesssim 
        \left(\sum_{F\in\F^{\p}_{h}}h_{F}\|\dot{\left(\frac{\p w_{h}}{\p \ell}\right)}\|_{L^{2}(F)}^{2}  \right)^{1/2}
        \left(\sum_{F\in\F^{\p}_{h}}h_{F}\|\dot{\left(\frac{\p v_{h}}{\p\ell^{\perp}}\right)}\|^{2}_{L^{2}(F)}\right)^{1/2}.
        \end{aligned}
\end{equation} 
For a boundary edge $F=\p K\cap\p\Omega$, combining Lemma \ref{lm:trace} (trace estimate), Lemma \ref{lm:inverse} (inverse estimate) and \eqref{eq:dis-mt-ell}, we have 
\begin{equation*} 
\begin{aligned}
    h_{F}\|\dot{\left(\frac{\p w_{h}}{\p \ell}\right)}\|_{L^{2}(F)}^{2}&\lesssim
    h_{F}\left(\|\nabla w_{h}\|_{L^{2}(F)}^{2}+\|D^{2}w_{h}\|_{L^{2}(F)}^{2}\right) ~~~~~~~~~~~~~~~~\text{(by \eqref{eq:dis-mt-ell})}\\ 
    &\lesssim |w_{h}|^{2}_{H^{1}(K)}+(h_{K}^{2}+1)|w_{h}|_{H^{2}(K)}^{2}+h_{K}^{2}|w_{h}|_{H^{3}(K)}^{2} ~\text{(Lemma \ref{lm:trace})}\\ 
    &\lesssim |w_{h}|^{2}_{H^{1}(K)}+|w_{h}|_{H^{2}(K)}^{2}.~~~~~~~~~~~~~~~~~~~~~~~~~~~~~~~~\text{(Lemma \ref{lm:inverse})}
\end{aligned}
\end{equation*} 
By Lemma \ref{lm:poin} (Poincar\'e-type estimate), we obtain 
\begin{equation}\label{eq:bound-Fp-1} 
    h_{F}\|\dot{\left(\frac{\p w_{h}}{\p \ell}\right)}\|_{L^{2}(F)}^{2} \lesssim 
    |w_{h}|_{H^{2}(K)}^{2} + \|\nabla w_{h}\|^{2}_{L^{2}(F)}.
\end{equation} 
Similarly $
h_{F}\|\dot{\left(\frac{\p v_{h}}{\p\ell^{\perp}}\right)}\|^{2}_{L^{2}(F)}\lesssim 
|v_{h}|_{H^{2}(K)}^{2} + \|\nabla v_{h}\|^{2}_{L^{2}(F)}.$
This, combined with \eqref{eq:bound-Fp-1}, gives $I_{3} \lesssim \|w_{h}\|_{h}\|v_{h}\|_{h}$. Then, the proof is concluded by combing the estimate for $I_{i}$, $i=1,2,3$. 
\end{proof}

\begin{theorem}
    \label{thm:num-wellpose}Under the Assumption \ref{ass:problem}, 
    let $u\in H^{2}_{\ell,0}(\Omega)$ be the $H^{2}$ strong solution of \eqref{eq:obl-der-problem}. 
    Let $\{\T_{h}\}_{h>0}$ be regular of order $2$. 
    Then, there exists a unique solution $u_{h} \in V_{h,0}^{\ell}$ to \eqref{eq:num-scheme} satisfying  
\begin{equation*}
\begin{aligned}
    &\left(\sum_{K\in\T_{h}}|u-u_{h}|_{H^{2}(K)}^{2}\right)^{1/2} \\ 
    &\leq \frac{C_{\rm{err}}}{1-\sqrt{1-\varepsilon}} \inf_{w_{h}\in V^{\ell}_{h,0} }
    \left(  
   \sum_{K\in\T_{h}}|u-w_{h}|^{2}_{H^{2}(K)} + h_{K}^{-1}\|\nabla (u-w_{h})\|^{2}_{L^{2}(\p K)}
    \right)^{1/2}.
\end{aligned}
\end{equation*} 
Here, the $C_{\rm{err}}$ depends on the the polynomial degree $k$, the shape-regular constant $\sigma$ and $c_{1}, c_{2}, c_{3}$.  
\end{theorem} 
\begin{proof}
    The existence and uniqueness of the numerical solution $u_{h}\in V^{\ell}_{h,0}$ follow from Lemma \ref{lm:coer} (coercivity), Lemma \ref{lm:boundedness} (boundedness), and the Lax-Milgram Theorem. For any $w_{h}\in V^{\ell}_{h,0}$, denote $\psi_{h}:=u_{h}-w_{h} \in V^{\ell}_{h,0}$. Lemma \ref{lm:coer} (coercivity) and Lemma \ref{lm:consistence} (consistence) lead to 
    $$ \begin{aligned}
        \|\psi_{h}\|^{2}_{h} &\leq \frac{1}{1-\sqrt{1-\varepsilon}}b_{h}(u_{h}-w_{h},\psi_{h})=\frac{1}{1-\sqrt{1-\varepsilon}}b_{h}(u-w_{h},\psi_{h}).
    \end{aligned} $$ 
    Here, we denote 
    $$ \begin{aligned}
        b_{h}(u-w_{h},\psi_{h}) & \lesssim 
        \underbrace{|\sum_{K\in\T_{h}}(\gamma A:D^{2}(u-w_{h}),\Delta \psi_{h})_{K}|}_{I_{1}} + 
        \underbrace{|\sum_{F\in\F^{i}_{h}}\int_{F}[\frac{\p (u-w_{h})}{\p n}]\frac{\p^{2} \psi_{h}}{\p t_{F}^{2}}\ds|}_{I_{2}} \\ 
        & + 
        \underbrace{|\sum_{F\in\F^{\p}_{h}}\int_{F}\dot{\left(\frac{\p (u-w_{h})}{\p \ell}\right)}\frac{\p \psi_{h}}{\p \ell^{\perp}} \ds|}_{I_{3}}.
    \end{aligned} $$
    For $I_{1}$, similar arguments as \textit{Step 1} in Lemma \ref{lm:boundedness} (boundedness) lead to 
    $$ I_{1}\lesssim 
    \left(\sum_{K\in\T_{h}}|u-w_{h}|^{2}_{H^{2}(K)}\right)^{1/2}\|\psi_{h}\|_{h}. $$
    For $I_{2}$, we have 
    $$ \begin{aligned}
        I_{2} &\leq 
        \left(\sum_{F\in\F^{i}_{h}}h_{F}^{-1}\|[\frac{\p (u-w_{h})}{\p n}]\|^{2}_{L^{2}(F)}\right)^{1/2}
        \left(\sum_{F\in\F^{i}_{h}}h_{F}\|\frac{\p^{2} \psi_{h}}{\p t_{F}^{2}}\|^{2}_{L^{2}(F)}\right)^{1/2} \\ 
        &\lesssim \left(\sum_{F\in\F^{i}_{h}}h_{F}^{-1}\|[\frac{\p (u-w_{h})}{\p n}]\|^{2}_{L^{2}(F)}\right)^{1/2}
        \|\psi_{h}\|_{h}.
    \end{aligned} $$
    In the last step, we use the same argument as \textit{Step 2} \eqref{eq:bound-Fi-1} in Lemma \ref{lm:boundedness} (boundedness). 

    For $I_{3}$, we first apply the integration by part on edge $F \in \F^{\p}_{h}$. Due to the $C^{1}$-continuity on vertices of the finite element space and the boundary condition of $u \in H^{2}_{\ell,0}(\Omega)$, we obtain
    $$
    \begin{aligned}
        I_{3} &= |\sum_{F\in\F^{\p}_{h}}\int_{F}\dot{\left(\frac{\p (u-w_{h})}{\p \ell}\right)}\frac{\p \psi_{h}}{\p \ell^{\perp}} \ds| = |\sum_{F\in\F^{\p}_{h}}\int_{F}\frac{\p (u-w_{h})}{\p \ell}\dot{\left(\frac{\p \psi_{h}}{\p \ell^{\perp}}\right)} \ds| \\ 
        &\leq 
        \left(\sum_{F\in\F^{\p}_{h}}h_{F}^{-1}\|\frac{\p (u-w_{h})}{\p \ell}\|^{2}_{L^{2}(F)}\right)^{1/2}
        \left(\sum_{F\in\F^{\p}_{h}}h_{F}\|\dot{\left(\frac{\p \psi_{h}}{\p \ell^{\perp}}\right)}\|^{2}_{L^{2}(F)}\right)^{1/2} \\
        &\lesssim 
        \left(\sum_{F\in\F^{\p}_{h}}h_{F}^{-1}\|\nabla (u-w_{h})\|^{2}_{L^{2}(F)}\right)^{1/2}
        \|\psi_{h}\|_{h}.
    \end{aligned}
    $$
    In the last step, we use the same argument as  \textit{Step 3} \eqref{eq:bound-Fp-1} in Lemma \ref{lm:boundedness} (boundedness). By combing the estimate of $I_{i}$, $i=1,2,3$, we arrive
    $$ \|\psi_{h}\|_{h} \lesssim  \left(  
        \sum_{K\in\T_{h}}|u-w_{h}|^{2}_{H^{2}(K)} + h_{K}^{-1}\|\nabla (u-w_{h})\|^{2}_{L^{2}(\p K)}\right)^{1/2}. $$
    This, combined with the triangular inequality, concludes the proof.
\end{proof}

Taking $w_{h} = \Pi_{h}^{\ell} u$ in Theorem \ref{thm:num-wellpose}, and combining with \eqref{eq:app-Pi-ell-F} and \eqref{eq:app-Pi-ell-K} in Theorem \ref{thm:app_Vh_ell} (approximation property of $V^{\ell}_{h,0}$), we have the following error estimate result. 
\begin{corollary}
    \label{corollary:err}Under the Assumption \ref{ass:problem}, 
    assume that $\{ \mathcal{T}_{h} \}_{h>0}$ is regular of order $k$, 
    $u \in H^{s}(\Omega) \cap H^{2}_{\ell,0}(\Omega)$ is the $H^{2}$ strong solution of \eqref{eq:obl-der-problem}, and the oblique vector field $\ell$ is piecewise $C^{s-1}$ with $2\leq s \leq k+1$. 
    It holds that
    \begin{equation}\label{eq:err} 
        \left(\sum_{K\in\T_{h}}|u-u_{h}|_{H^{2}(K)}^{2}\right)^{1/2} 
        \lesssim \frac{1}{1 - \sqrt{1-\varepsilon}}
        \left(\sum_{K\in\T_{h}}h_{K}^{2s-4}\|u\|^{2}_{H^{s}(\omega_{K})}\right)^{1/2},
    \end{equation} 
    where the hidden constant depends on $\ell$, the polynomial degree $k$, the shape-regular constant $\sigma$ and $c_{1}, c_{2}, \cdots, c_{k+1}$. 
\end{corollary}

\section{Numerical experiments} \label{sec:numerical}
This section presents numerical experiments of the proposed methods for \eqref{eq:obl-der-problem}. 
We apply the curved Hermite element discussed in Section \ref{sec:fem} with polynomial degree $k = 3$.
The convergence history plots are logarithmically scaled in all the convergence order experiments.

\subsection{Experiment 1}\label{exp:1}
In the first experiment, we consider the problem \eqref{eq:obl-der-problem} in the unit disk $\Omega := \{|x|^{2} < 1\}$.  
The coefficient matrix is set to be
$ A := \begin{pmatrix} 
    1 & 0 \\
    0 & 1 
    \end{pmatrix}. $ 
Take the oblique vector field $\ell$ to be the rotated normal vector, i.e., 
    $$ \ell := \sqrt{1/2} \begin{pmatrix}
          1 & -1 \\ 
          1 & 1 \\
         \end{pmatrix} n, $$
so that the rotation angle is $\theta = \pi/4$.     
To test the convergence rate, we consinder the smooth solution
$$ u = \sin(\pi(x_{1}^2 + x_{2}^2))e^{(x^{2}_{1}+x_{2}^{2})} - \frac{\pi(e+1)}{\pi^{2}+1}. $$
The function $f := A:D^{2}u$ is directly calculated from the coefficient matrix and solution.
A straightforward calculation shows that the compatibility constant $c = -\sqrt{2}\pi e$ and 
$\chi_{0} = 1$.
 
We apply the numerical scheme \eqref{eq:num-scheme} to the problem on a sequence of uniform triangulations $\{\mathcal{T}_{h}\}_{h>0}$. 
The discretization errors in the $L^{2}$, $H^{1}$ and $H^{2}$ norm are reported in Table \ref{tab:exp1}. 
We also display the approximations of the compatibility constant.  
The expected optimal convergence rate $\|D^{2}(u - u_{h})\|_{L^{2}(\mathcal{T}_{h})} = \mathcal{O}(h^{k-1})$ is observed, agreeing with Corollary \ref{corollary:err}. Further, the experiments indicate that the scheme converges with (sub-optimal) second-order convergence in both $H^{1}$ and $L^{2}$ norm. 
Convergence to the compatibility constant $c=-\sqrt{2}\pi e$ is also observed for the approximation $c_{h}$. 
\begin{table}[!htbp]
    \centering
    \begin{tabular}{llllllll}
    \toprule
      $h$  &$\|u - u_{h}\|_{L^{2}}$ & Order&$|u-u_h|_{H^{1}}$& Order&$|u-u_h|_{H^{2}}$& Order &$c_{h}$
    \\
    \midrule 
    $2^{-2}$ &  2.80e-02 &   0.00  &  2.88e-01  &  0.00    &4.95 &   0.00 &-12.128\\
    $2^{-3}$ &   7.67e-03 &   1.87 &  6.84e-02  &  2.07   & 2.81   & 0.82&-12.086\\
    $2^{-4}$ &   1.43e-03 &   2.43 &  1.20e-02   & 2.51   & 8.48e-01 &   1.73&-12.077\\
    $2^{-5}$ &   2.97e-04 &   2.26 &  2.36e-03   & 2.35   & 2.21e-01  &  1.94&-12.077\\
    $2^{-6}$ &   7.15e-05 &   2.05  &  5.58e-04   & 2.08  &  5.55e-02  &  1.99&-12.077 \\
    \bottomrule
    \end{tabular}
    \caption{Errors and observed convergence orders for Experiment 1.}\label{tab:exp1}
\end{table}

\subsection{Experiment 2}\label{exp:2}
In the second experiment, let the domain $\Omega := \{|x|^{2} < 1\}$. The coefficient matrix is set to be
$$ A := \begin{pmatrix} 
    2 & \frac{x_1 x_2}{|x_1 x_2|} \\
    \frac{x_1 x_2}{|x_1 x_2|} & 2 
    \end{pmatrix}. $$ 
A straightforward calculation shows the Cordes condition \eqref{eq:cordes} is satisfied with $\varepsilon = 3/5$. We note that the coefficient matrix is discontinuous across the set 
$\{(x_1, x_2) \in \Omega:  x_{1} = 0 \text{ or } x_{2} = 0\}.$
Take the oblique vector field $\ell$ to be the anti-clockwise rotation of the normal vector by the angle $\theta = s + \pi/4$, i.e., 
    $$ \ell := (\cos(2s + \pi/4), \sin(2s + \pi/4))^{T}, $$ where $s := \arctan(x_{2}/x_{1})$. The function $f$ is chosen so that the solution is given by 
$$ u = 1/6(x_{1}^2 + x_{2}^2)^3 - 1/2 (x_{1}^2 + x_{2}^2) + 5/24. $$
Directly calculation shows that the compatibility constant $c = 0$ and $\chi_{0} = 2$. 

The numerical result is displayed in Figure \ref{fig:exp2}, which is in agreement with Corollary \ref{corollary:err}. The convergence orders are similar to Experiment \ref{exp:1}. 
This experiment demonstrates the robustness of this proposed scheme with respect to the choice of oblique vector field and the discontinuous coefficient matrix. 

\begin{figure}[!htbp]
  \centering
  \captionsetup{justification=centering}
  \includegraphics[width=0.52\textwidth]{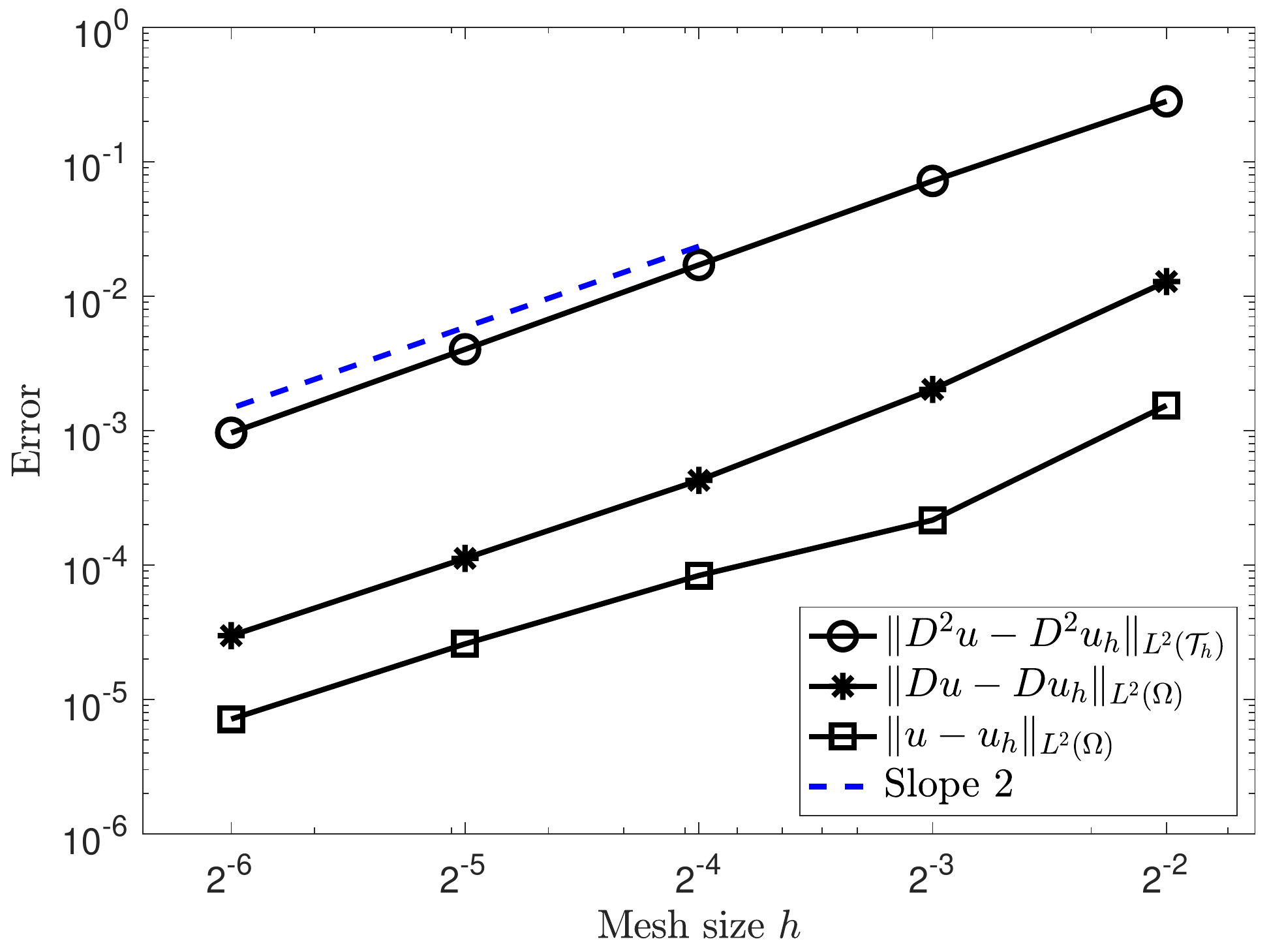}
  \caption{Convergence orders for Experiment 2.}
  \label{fig:exp2}
  \end{figure}


\subsection{Experiment 3}\label{exp:3}
In the third experiment, let the domain $\Omega := \{|x|^{2} < 1\}$. 
The coefficient matrix $A$ is the same as in Experiment \ref{exp:2}. 
The oblique vector field is given as 
$$ \ell := \sqrt{1/2} \begin{pmatrix}
      1 & -1 \\ 
      1 & 1 \\
     \end{pmatrix} n, $$
so that the rotation angle is $\theta = \pi/4$. 
Take $f$ so that the solution is  
$$ u = (x^{2}_{1}+x_{2}^{2})e^{(x^{2}_{1}+x_{2}^{2})} - 1. $$ 
Here the compatibility constant $c = 2\sqrt{2}e$ and $\chi_{0} = 1$.
The numerical results on uniformly refined meshes are shown in Table \ref{tab:exp3}. 
Similar convergence orders to Experiments \ref{exp:1} and \ref{exp:2} are observed. The convergence to the compatibility constant $c = 2\sqrt{2}e$ is also observed.

\begin{table}[!htbp]
    \centering
    \begin{tabular}{llllllll}
    \toprule
      $h$  &$\|u - u_{h}\|_{L^{2}}$ & Order&$|u-u_h|_{H^{1}}$& Order&$|u-u_h|_{H^{2}}$& Order& $c_{h}$
    \\
    \midrule
  $2^{-2}$&  9.29e-03&   0.00&   1.11e-01&   0.00&   2.52&   0.00 & 7.6997\\
  $2^{-3}$&   1.48e-03&   2.65&   1.80e-02&   2.62&   6.95e-01&   1.86&7.6882 \\
  $2^{-4}$&   5.51e-04&   1.43&   3.40e-03&   2.41&   1.70e-01&   2.03&7.6880 \\
  $2^{-5}$&   1.73e-04&   1.67&   8.25e-04&   2.04&   4.04e-02&   2.08&7.6883 \\
  $2^{-6}$&   4.81e-05&   1.85&   2.15e-04&   1.94&   9.64e-03&   2.07&7.6884 \\
    \bottomrule
    \end{tabular}
    \caption{Errors and observed convergence orders for Experiment 3.}\label{tab:exp3}
\end{table}


\subsection{Experiment 4}\label{exp:4}
In the fourth experiment, let $\Omega := \{x_{1}^{2}/4 + x_{2}^{2} < 1\}$ be an ellipse domain.  The coefficient matrix is the same as in Experiment \ref{exp:2}.
The oblique vector field $\ell$ is taken to be the tangential vector.
We take the function $f$ so that the solution is given by
 $ u = 1/4 \sin(\pi(x_{1}^2/4 + x_{2}^2)) - 1/(2\pi).$ In this case, the compatibility constant $c = 0$ and $\chi_{0} \geq 1/4$.
We apply the numerical scheme \eqref{eq:num-scheme} to the problem and refine the mesh uniformly. Similar convergence orders to the above experiments are observed. Numerical results and the mesh are shown in Figure \ref{fig:exp4}.

\begin{figure}[!htbp]
  \centering 
  \captionsetup{justification=centering}
  \subfloat[Convergence orders.]{
    \includegraphics[width=0.43\textwidth]{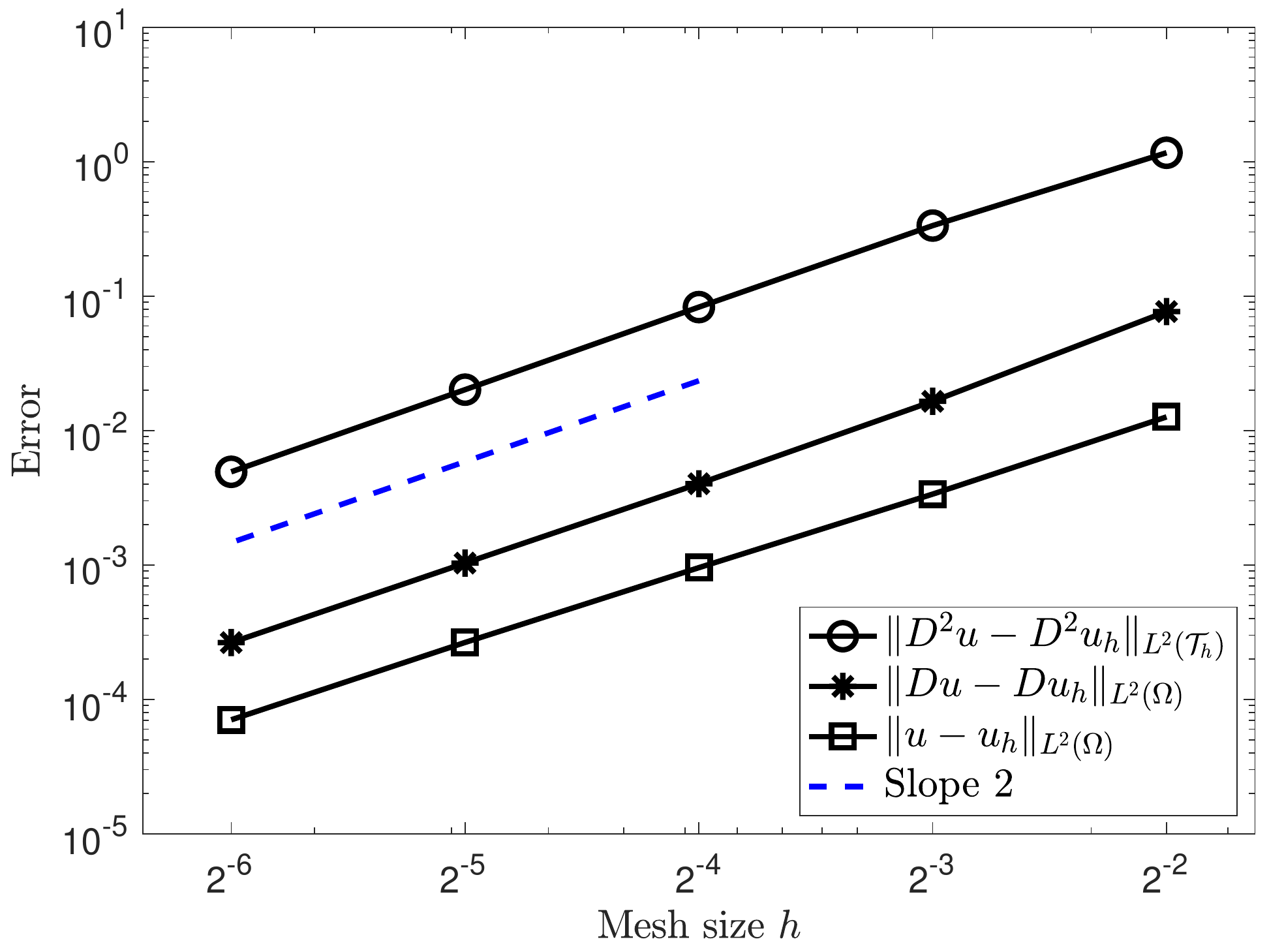}
    \label{fig:test4-k3}
  } 
  \subfloat[Mesh on an ellipse domain with $h=2^{-3}$.]{
    \includegraphics[width=0.43\textwidth]{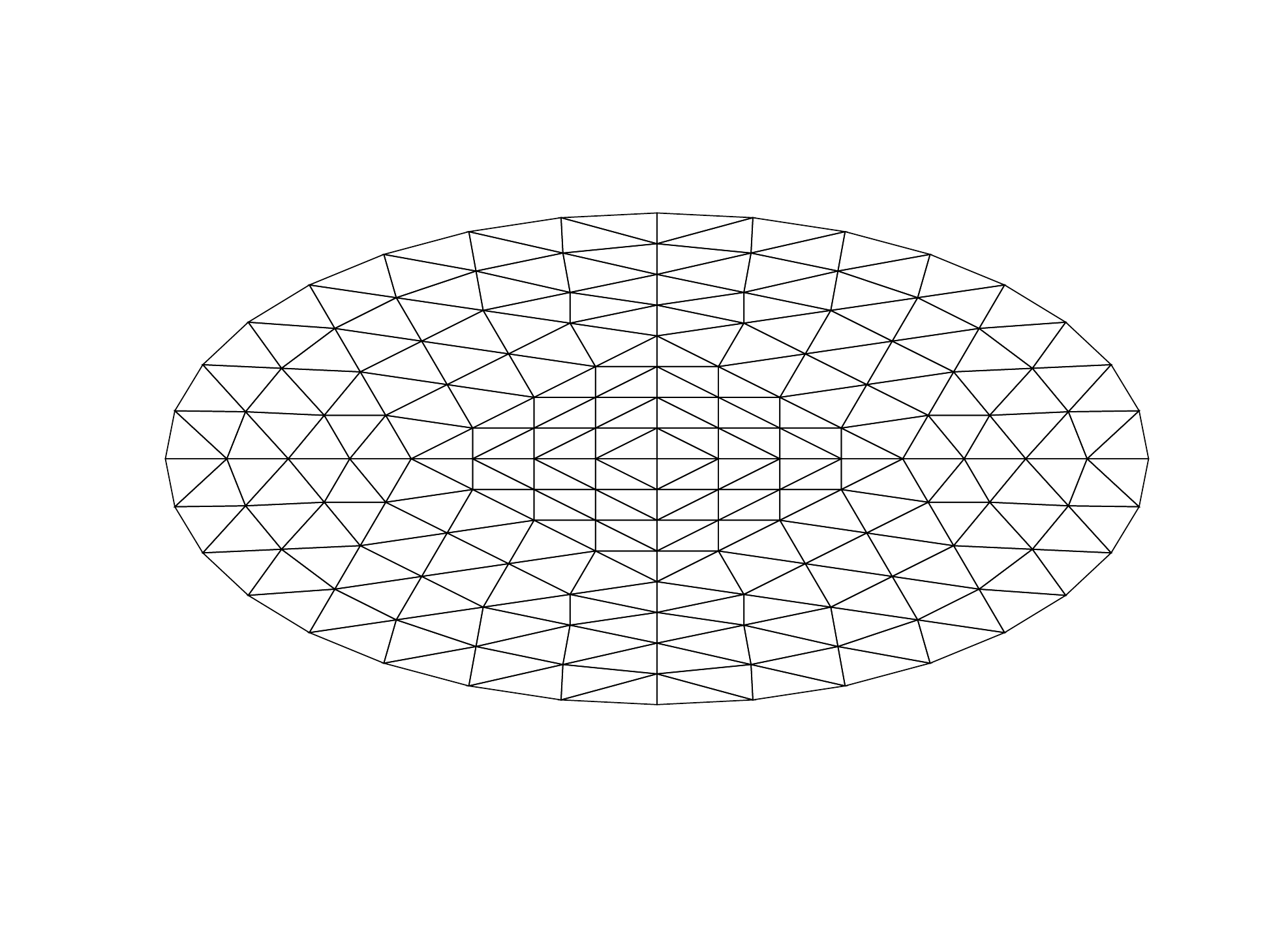}
    \label{fig:test4-mesh}
  } %
  \caption{Convergence orders and mesh for Experiment 4.}
  \label{fig:exp4}
  \end{figure}


\appendix

\section{Approximation property of finite element space}\label{app:approximation} 
This appendix considers the construction of the quasi-interpolation operator $\Pi_{h}^{\ell}: H^{2}_{\ell,0}(\Omega) \rightarrow V_{h,0}^{\ell}$ 
and its optimal approximation property. 

\subsection{Construction of $\Pi_{h}^{\ell}$}

We consistently use the following correspondence in the sequel: $\boldsymbol{x} = F_{K}(\hat{\boldsymbol{x}})$, $v = \hat{v} \circ F^{-1}_{K}$.
Let $N_{h} := \dim V_{h}$ and $\mu_{i}(\cdot)$, $1\leq i \leq N_{h}$, be the global degrees of freedom of $V_{h}$. 
Denote $\bo{a}_{i}$, the node where $\mu_{i}(\cdot)$ takes its value. Let $s_{i}$ equal to $0$ or $1$, indicating whether $\mu_{i}(\cdot)$ takes a function or derivative value.
Denote $\phi_{i} \in V_{h}$, the basis function of $\mu_{i}(\cdot)$.  
$\omega_{i} := 
\bigcup\{K:\mathring{K}\cap\operatorname{supp}(\phi_{i}) \neq \varnothing, K\in \T_{h}\}$, and 
$\# \omega_{i}: = 
\#\{K:\mathring{K}\cap\operatorname{supp}(\phi_{i}) \neq \varnothing, K \in \T_{h}\}$. 
Note that $\# \omega_{i}$ is uniformly bounded if $\{\T_{h}\}_{h>0}$ is shape-regular. 
For any $K \in \T_{h}$, $\omega_{K}$ is the union of all $\omega_{i}$ containing $K$.

\begin{lemma}[scale of $\phi_{i}$]
  \label{lm:scale-phi-mu}Assume $\{\mathcal{T}_{h}\}_{h>0}$ is regular of order $k$. 
  Then for $1 \leq i \leq N_{h}$, 
\begin{equation}\label{eq:scale-phi}
  \|\phi_{i}\|_{H^{m}(K)} \lesssim h_{K}^{1 - m + s_{i}} \qquad 0\leq m \leq k+1,
\end{equation}
where $K \in \mathcal{T}_{h}$ is contained in $\omega_{i}$ and the hidden constant is depends on the polynomial degree $k$, the shape-regular constant $\sigma$ and $c_{1}, c_{2}, \cdots, c_{k+1}$.
\end{lemma} 
\begin{proof}
\underline{\textit{Case 1: $s_{i} = 0$.}} In this case, $\mu_{i}(v) = v(\bo{a}_{i})$. 
By Lemma \ref{lm:scaling} (scaling on curved triangles), we have 
$ \|\phi_{i}\|_{H^{m}(K)} \lesssim h_{K}^{1 - m}\|\hat{\phi}_{i}\|_{H^{m}(\hat{K})}. $
It is easy to check that $\hat{\phi}_{i} = \phi_{i} \circ F_{K}$ is a basis function of $(\hat{K}, \mathcal{P}_{k}(\hat{K}), \hat{\Sigma})$, hence $\|\hat{\phi}_{i}\|_{H^{m}(\hat{K})} = \mathcal{O}(1)$.

\underline{\textit{Case 2: $s_{i} = 1$}.}
In this case, $\mu_{i}(v) = \partial_{j} v(\bo{a}_{i})$, where $j = 1$ or $2$ and $\bo{a}_{i} \in \mathcal{N}_{h}$ is a vertex.  
The chain rule shows that $\hat{\phi}_{i} = \phi_{i} \circ F_{K}$ vanishes on all the degrees of freedom in $\hat{\Sigma}$ except 
$$ \hat{\partial}_{j} \hat{\phi}_{i}(\hat{\bo{a}_{i}}) 
= \partial_{1} \phi_{i}(\bo{a}_{i}) (DF_{K}(\hat{\bo{a}_{i}}))_{1,j} 
+ \partial_{2} \phi_{i}(\bo{a}_{i}) (DF_{K}(\hat{\bo{a}_{i}}))_{2,j} \quad j = 1,2.  $$
Then, the norm equivalence on $\mathcal{P}_{k}(\hat{K})$ leads to
\begin{equation*} 
  \begin{aligned}
      \|\hat{\phi}_{i}\|^{2}_{H^{m}(\hat{K})} \eqsim 
      \sum_{\hat{\mu}\in\hat{\Sigma}} \hat{\mu}(\hat{\phi}_{i})^{2}  
      &=(\hat{\partial}_{1}\hat{\phi}_{i}(\hat{\bo{a}}_{i}))^{2} 
            + (\hat{\partial}_{2}\hat{\phi}_{i}(\hat{\bo{a}}_{i}))^{2} \\ 
      &\leq\sup_{\hat{x}\in\hat{K}}\|DF_{K}(\hat{x})\|^{2}
      \left(({\partial}_{1}{\phi}_{i}({\bo{a}_{i}}))^{2}
           +({\partial}_{2}{\phi}_{i}({\bo{a}_{i}}))^{2}\right) \lesssim h_{K}^{2}.  
  \end{aligned}
\end{equation*} 
This, combined with the scaling argument, yields
$\|\phi_{i}\|_{H^{m}(K)}  \lesssim h_{K}^{2-m}$. 
\end{proof}

Let $u \in L^{2}(K)$ and $\hat{u} := u \circ F_{K}$. Denote $Q_{\hat{K}}\hat{u}$ the $L^{2}$ projection of $\hat{u}$ on $\mP_{k}(\hat{K})$. 
Then, we define the operator $Q_{K}: L^{2}(K) \rightarrow P_{K}$ by 
$$ Q_{K}u := Q_{\hat{K}}\hat{u} \circ F_{K}^{-1}. $$ 
By the Lemma \ref{lm:scaling} (scaling on curved triangles) and the the Bramble-Hilbert Lemma \cite{brenner2007mathematical}, we have the following approximation property of $Q_K$.
\begin{lemma}[approximation property of $Q_{K}$]
  \label{lm:QK-app}Let $\{\T_{h}\}_{h>0}$ is of regular of $k$, $u \in H^{s}(\Omega)$ with $2 \leq s \leq k + 1$. 
Then we have   
\begin{equation}\label{eq:QK-app} 
    \|u - Q_{K}u\|_{H^{m}(K)} \lesssim  h_{K}^{s - m} \|u\|_{H^{s}(K)}, \qquad m=0,1,2.
\end{equation}  
Here the hidden constant depends on the polynomial degree $k$, the shape-regular constant $\sigma$ and $c_{1}, c_{2}, \cdots, c_{k+1}$.
\end{lemma} 

We are now ready to define the quasi-interpolation operator $\Pi_{h}: L^{2}(\Omega) \rightarrow V_{h}$ by averaging the local $L^{2}$ projection, i.e., 
\begin{equation}\label{def:Pi_h} 
  \mu_{i}(\Pi_{h}u) := \frac{1}{\# \omega_{i}}
  \sum^{\# \omega_{i}}_{j=1}\mu_{i}(Q_{K_{j}}u), \quad 1 \leq i \leq N_{h}.
\end{equation} 
Here, we write $ \omega_{i} = \bigcup^{\# \omega_{i}}_{j=1} K_{j}$ with $K_{j}$ and $K_{j+1}$ sharing a common edge.  
\begin{theorem}[approximation property of $\Pi_{h}$]
  \label{thm:app_Vh}Assume $\{\T_{h}\}_{h>0}$ is of regular of order $k$. 
  Let $u \in H^{s}(\Omega)$ with $2 \leq s \leq k+1$. Then we have 
  \begin{equation}\label{eq:app-pi} 
  \|u - \Pi_{h}u\|_{H^{m}(K)} \lesssim h_{K}^{s - m} \|u\|_{H^{s}(\omega_{K})}, \qquad m=0,1,2.
  \end{equation} 
  Here, the hidden constant depends on the polynomial degree $k$, the shape-regular constant $\sigma$ and $c_{1}, c_{2}, \cdots, c_{k+1}$. 
\end{theorem} 
\begin{proof}
  For now, we number $\omega_{1}, \cdots, \omega_{n}$ the sets $\omega_{i}$ such that $K \subset \omega_{i}$. Then, we have 
  $$ \|Q_{K}u - \Pi_{h}u\|_{H^{m}(K)} \leq \sum^{n}_{i=1} 
  |\mu_{i}(Q_{K}u) - \mu_{i}(\Pi_{h}u)| \|\phi_{i}\|_{H^{m}(K)}.$$
In the case of $\# \omega_{i} = 1$, we know that $\omega_{i} = K$ and $\mu_{i}(Q_{K}u) - \mu_{i}(\Pi_{h}u) = 0 $. Below, we assume $\# \omega_{i} \geq 2$. By \eqref{def:Pi_h} and triangle inequality, we have
\begin{equation}\label{eq:app-Vh-dof-jump} 
\begin{aligned}
  &\big| \mu_{i}(Q_{K}u) - \mu_{i}(\Pi_{h}u)  \big|
  = \left|
  \frac{1}{\# \omega_{i}}\sum^{\# \omega_{i}}_{j=1}\left( 
    \mu_{i}(Q_{K}u)-\mu_{i}(Q_{K_{j}}u)\right)
  \right| \\ 
  &\lesssim \sum^{\# \omega_{i}-1}_{j=1}\left|\mu_{i}(Q_{K_{j}}u) - \mu_{i}(Q_{K_{j+1}}u)\right| 
  \lesssim \sum_{\substack{F\in\F_{h}^{i}\\ \text{with }\bo{a}_{i} \in \overline{F} }}
  |Q_{K^{+}}u - Q_{K^{-}}u|_{W^{s_{i}}_{\infty}(F)},
\end{aligned}
\end{equation}
where we denote $F = K^{+} \cap K^{-}$, and recall $\bo{a}_{i}$ is the node where $\mu_{i}$ takes its value, $s_{i} = 0$ or $1$. 
 Note $(Q_{K^{\pm}}u)\big|_{F} = Q_{\hat{K}}\hat{u} \circ (F_{K^{\pm}}^{-1}\big|_{F})$, where 
 $F_{K^{\pm}}^{-1}\big|_{F} $ is an affine mapping (see Remark \ref{rem:construction-F_K}). Therefore, it is clear that $(Q_{K^{\pm}}u)\big|_{F}$ is indeed a polynomial. Hence, we have  
 \begin{equation}\label{eq:app-Vh-jump1} 
  \begin{aligned}
    |Q_{K^{+}}u  - Q_{K^{-}}u|_{W^{s_{i}}_{\infty}(F)} &\lesssim 
    h_{F}^{-\frac{1}{2}-s_{i}}\|Q_{K^{+}}u - Q_{K^{-}}u\|_{L^{2}(F)} \\ 
    &\leq h_{F}^{-\frac{1}{2}-s_{i}}\left(
      \|Q_{K^{+}}u - u\|_{L^{2}(F)} + \|Q_{K^{-}}u - u\|_{L^{2}(F)}
    \right). 
   \end{aligned}
 \end{equation} 
Lemma \ref{lm:trace} (trace estimate) and Lemma \ref{lm:QK-app} (approximation property of $Q_{K}$) lead to 
\begin{equation}\label{eq:app-Vh-jump2}
\begin{aligned}
  \|Q_{K^{\pm}}u - u\|_{L^{2}(F)} & \lesssim 
  \left(h_{K^{\pm}}^{-\frac{1}{2}}\|Q_{K^{\pm}}u - u\|_{L^{2}(K^{\pm})} 
  + h_{K^{\pm}}^{\frac{1}{2}}|Q_{K^{\pm}}u - u|_{H^{1}(K^{\pm})}\right)\\ 
  & \lesssim h_{K}^{s-\frac{1}{2}}\|u\|_{H^{s}(\omega_{K})}.  
\end{aligned}
\end{equation} 
Combing the above estimates, we obtain 
$\big| \mu_{i}(Q_{K}u) - \mu_{i}(\Pi_{h}u)\big| 
\lesssim h_{K}^{s-1-s_{i}}\|u\|_{H^{s}(\omega_{K})}.  $
This, combined with Lemma \ref{lm:scale-phi-mu} (scale of $\phi_{i}$), leads to $\|Q_{K}u - \Pi_{h}u\|_{H^{m}(K)}\lesssim h_{K}^{s-m}\|u\|_{H^{s}(\omega_{K})}.$  Then the proof is concluded by triangle inequality and Lemma \ref{lm:QK-app} (approximation property of $Q_{K}$).
\end{proof}

Next, we reorder the degrees of freedom such that for 
$i = N_{0} + 1, \cdots, N_{h}$, $\mu_{i}(v)=(\nabla v\cdot\bo{e}_{i}) (\bo{a}_{i}), $
where $\bo{a}_{i} \in \N^{\p}_{h}$ is a boundary vertex and $\bo{e}_{i} = (1,0)^{T}$ or $(0,1)^{T}$. 
Then for any $u\in H^{2}_{\ell,0}(\Omega)$, denote $C_{u} := (\nabla u\cdot \ell)|_{\p\Omega}$.
We first define $w_{h} \in V_{h}$ as follows. 
\begin{subequations} \label{def:Pi_h^ell} 
  \begin{itemize}
    \item For $i = 1,\cdots,N_{0}$, 
\begin{equation} \label{eq:def-wh-inner}
  \mu_{i}(w_{h}) : = \mu_{i}(\Pi_{h}u) = \frac{1}{\# \omega_{i}}
  \sum^{\# \omega_{i}}_{j=1}\mu_{i}(Q_{K_{j}}u).  
\end{equation}
\item For $i = N_{0} + 1, \cdots, N_{h}$,
\begin{equation} \label{eq:def-wh-bd}
  \mu_{i}(w_{h}) := \left(
    (\bo{e}_{i}\cdot\ell)C_{u}+
    (\bo{e}_{i}\cdot\ell^{\perp})\frac{1}{\# \omega_{i}}\sum^{\# \omega_{i}}_{j=1}
    \frac{\p Q_{K_{j}}u}{\p \ell^{\perp}}      
    \right)(\bo{a}_{i}).
\end{equation}
\end{itemize}
It can be easily verified that $ (\nabla w_{h} \cdot \ell)(\bo{x}) = C_{u} $ for any $\bo{x}\in\N^{\p}_{h}$. 
Finally, the quasi-interpolation operator $\Pi^{\ell}_{h}: H^{2}_{\ell,0}(\Omega) \rightarrow V^{\ell}_{h,0}$ is defined by 
\begin{equation}\label{eq:def-Pi_h^ell} 
  \Pi^{\ell}_{h}u := w_{h} - \frac{1}{|\Omega|}\int_{\Omega} w_{h} \dx.
\end{equation} 
\end{subequations}

\subsection{Proof of Theorem \ref{thm:app_Vh_ell}}
   By \eqref{eq:def-Pi_h^ell}, we know $|\Pi_{h}u-\Pi_{h}^{\ell}u|_{H^{m}(K)} = |\Pi_{h}u-w_{h}|_{H^{m}(K)}, $ for $m = 1,2$. 
  Note that if $K$ does not contain boundary vertex, then $ (\Pi_{h}u-w_{h})|_{K} \equiv 0$ by \eqref{eq:def-wh-inner}. 
  Below, we assume $K$ contains a boundary vertex. We number $\omega_{1}, \cdots,\omega_{n}$ the sets $\omega_{i}$ satisfying $K \subset \omega_{i}$ and $N_{0}+1\leq i\leq N_{h}$. Then,
  \begin{equation*}\label{eq:app-Vhl-err-dofbd} 
    \begin{aligned}
      |\Pi_{h}u-w_{h}|_{H^{m}(K)} 
    & \leq \sum^{n}_{i=1}
    \big|\mu_{i}(\Pi_{h}u) - \mu_{i}(w_{h})\big| |\phi_{i}|_{H^{m}(K)}.
    \end{aligned} 
  \end{equation*} 
For any $1 \leq i \leq n$, there exists $K^{(i)} \subset \omega_{i} $ with edge $F^{(i)}$ satisfying $\bo{a}_{i} \in F^{(i)} \subset \p\Omega$.
Then, \eqref{eq:def-wh-bd} and triangle inequality leads to 
\begin{equation*}\label{eq:app-Vhl-dof-jump} 
\begin{aligned}
    \big|\mu_{i}(\Pi_{h}u) - &\mu_{i}(w_{h})\big|
   \leq \big|\frac{1}{\# \omega_{i}}\sum^{\# \omega_{i}}_{j=1}
  \frac{\p Q_{K_{j}}u}{\p \ell}(\bo{a}_{i}) - C_{u}\big|  \\ 
  & \lesssim 
  \underbrace{\sum^{\# \omega_{i}-1}_{j=1}\big| 
  \frac{\p Q_{K_{j}}u}{\p \ell}(\bo{a}_{i}) - \frac{\p Q_{K_{j+1}}u}{\p \ell}(\bo{a}_{i}) 
  \big|}_{I_{1}}
  +
  \underbrace{\big|\frac{\p Q_{K^{(i)}}u}{\p \ell}(\bo{a}_{i}) - C_{u}\big|}_{I_{2}}.
\end{aligned}
\end{equation*} 
Similar arguments as \eqref{eq:app-Vh-jump1} and \eqref{eq:app-Vh-jump2} in Theorem \ref{thm:app_Vh} (approximation property of $\Pi_{h}$) yield 
  $$ I_{1} \lesssim\sum_{\substack{F\in\F_{h}^{i}\\ \text{with }\bo{a}_{i} \in \overline{F} }}
  |Q_{K^{+}}u - Q_{K^{-}}u|_{W^{1}_{\infty}(F)}\lesssim h^{s-2}_{K}\|u\|_{H^{s}(\omega_{K})}.$$

For the estimate of $I_{2}$, let $\hat{F} := F_{K^{(i)}}^{-1}(F^{(i)})$ with vertex $\hat{\bo{a}} := F_{K^{(i)}}^{-1}(\bo{a}_{i})$. Define $\hat{\ell} : \hat{F} \rightarrow \mathbb{R}^{2}$ and $\hat{\eta} : \hat{F} \rightarrow \mathbb{R}^{2}$ by 
\begin{equation} \label{eq:eta-ell}
\hat{\ell}(\hat{x}) := (\ell \circ F_{K^{(i)}})(\hat{x}) 
\quad \text{ and } \quad 
\hat{\eta}(\hat{x}) := (DF_{K^{(i)}}(\hat{x}))^{-1} \hat{\ell}(\hat{x}) \quad \hat{x} \in \hat{F}.
\end{equation}
The regularity of $\{\T_{h}\}_{h>0}$ and $\ell$ implies that $\hat{\eta} \in C^{s-1}(\hat{F},\mathbb{R}^{2})$.
Now, pulling back $I_{2}$ to $\hat{F}$, and applying the chain rule yields 
\begin{subequations} \label{eq:Fhat-ell}  
  \begin{equation} 
    I_{2} = \big|(\nabla Q_{K^{(i)}}u\cdot\ell)(\bo{a}_{i}) - C_{u}\big| = 
    \big|(\hat{\nabla} Q_{\hat{K}}\hat{u}\cdot \hat{\eta})(\bh{a})-C_{u}\big|, 
  \end{equation}
  \text{and}
  \begin{equation} 
   C_{u} = (\nabla u \cdot \ell)\big|_{F^{(i)}} = 
      (\hat{\nabla}\hat{u}\cdot \hat{\eta})\big|_{\hat{F}}. 
  \end{equation}
  \end{subequations}
Let $\hat{\mathcal{I}} \hat{\eta} \in \mP_{k}(\hat{F}; \mathbb{R}^{2})$ be the Lagrange interpolation of $\hat{\eta}$ satisfying $ \hat{\mathcal{I}}\hat{\eta}(\bh{a}) = \hat{\eta}(\bh{a}). $ By \eqref{eq:Fhat-ell} and the norm equivalence on the polynomial space, we arrive at
  $$ \begin{aligned}
    I_{2} & = \big|(\hat{\nabla} Q_{\hat{K}}\hat{u}\cdot\hat{\eta})(\bh{a})-C_{u}\big| = 
       \big|(\hat{\nabla} Q_{\hat{K}}\hat{u}\cdot \hat{\mathcal{I}} \hat{\eta})(\bh{a})-C_{u}\big| \\ 
       & \lesssim \|\hat{\nabla} Q_{\hat{K}}\hat{u}\cdot \hat{\mathcal{I}} \hat{\eta}-C_{u}\|_{L^{2}(\hat{F})}
       =  \|\hat{\nabla} Q_{\hat{K}}\hat{u}\cdot \hat{\mathcal{I}}\hat{\eta}- \hat{\nabla}\hat{u}\cdot\hat{\eta}\|_{L^{2}(\hat{F})} \\ 
       & \leq \underbrace{\|\hat{\nabla} Q_{\hat{K}}\hat{u} - \hat{\nabla}\hat{u}\|_{L^{2}(\hat{F})} \|\hat{\mathcal{I}} \hat{\eta}\|_{L^{\infty}(\hat{F})}}_{J_{1}} + 
       \underbrace{\|\hat{\nabla}\hat{u}\|_{L^{2}(\hat{F})}\|\hat{\mathcal{I}}\hat{\eta}-\hat{\eta}\|_{L^{\infty}(\hat{F})}}_{J_{2}}. 
  \end{aligned} $$ 
  For $J_{1}$, we have 
  $ \|\hat{\mathcal{I}} \hat{\eta}\|_{L^{\infty}(\hat{F})} \leq \|\hat{\eta}\|_{L^{\infty}(\hat{F})} = \|(DF_{K^{(i)}})^{-1}\hat{\ell}\|_{L^{\infty}(\hat{F})} \lesssim h_{K^{(i)}}^{-1} \|\ell\|_{L^{\infty}(F^{(i)})}. $
  Lemma \ref{lm:trace} (trace estimate), Lemma \ref{lm:scaling} (scaling on curved triangles),  and Lemma \ref{lm:QK-app} (approximation property of $Q_{K}$) give
  $$ \begin{aligned}
    \|\hat{\nabla} Q_{\hat{K}}\hat{u} - \hat{\nabla}\hat{u}\|_{L^{2}(\hat{F})} & \lesssim 
    |Q_{\hat{K}}\hat{u} - \hat{u}|_{H^{1}(\hat{K})} + |Q_{\hat{K}}\hat{u} - \hat{u}|_{H^{2}(\hat{K})} \\ 
    &\lesssim \|Q_{K^{(i)}}u - u\|_{H^{1}(K^{(i)})} + h_{K^{(i)}}\|Q_{K^{(i)}}u - u\|_{H^{2}(K^{(i)})} \\ 
    & \lesssim h_{K^{(i)}}^{s-1}\|u\|_{H^{s}(\omega_{K})},
  \end{aligned} $$
which leads to $J_{1} \lesssim h_{K^{(i)}}^{s-2}\|u\|_{H^{s}(\omega_{K})}$. 

Similarly, for $J_{2}$, we have 
$$ \|\hat{\nabla} \hat{u}\|_{L^{2}(\hat{F})} \lesssim 
  |\hat{u}|_{H^{1}(\hat{K})} + |\hat{u}|_{H^{2}(\hat{K})} \lesssim \|u\|_{H^{2}(K^{(i)})} \leq \|u\|_{H^{s}(\omega_{K})}. $$
Recall the definition of $\hat{\eta}$ in \eqref{eq:eta-ell}, the Leibniz rule leads to 
$$ 
|\hat{\eta}|_{W^{s-1}_{\infty}(\hat{F})} \lesssim \sum_{j=0}^{s-1} |(DF_{K^{(i)}}(\hat{x}))^{-1}|_{W^{j}_{\infty}(\hat{F})}  |\hat{\ell}(\hat{x})|_{W^{s-1-j}_{\infty}(\hat{F})}.
$$ 
By the fact that $\ell$ is piecewise $C^{s-1}$ and $\{\T_{h}\}_{h>0}$ is regular of order $k$, the standard scaling argument on the curved element 
(cf. \cite[Lemma 2.2, Lemma 2.3]{bernardi1989optimal}) gives 
$$ 
 |(DF_{K^{(i)}}(\hat{x}))^{-1} |_{W^{j}_{\infty}(\hat{F})} \lesssim h_{F^{(i)}}^{j-1} \quad \text{and} \quad
 |\hat{\ell}(\hat{x})|_{W^{s-1-j}_{\infty}(\hat{F})} \lesssim h_{F^{(i)}}^{s-1-j} \|\ell\|_{W^{s-1-j}_{\infty}(F^{(i)})}.
$$ 
Therefore, we have $\|\hat{\eta} - \hat{\mathcal{I}} \hat{\eta}\|_{L^{\infty}(\hat{F})} \lesssim |\hat{\eta}|_{W^{s-1}_{\infty}(\hat{F})} \lesssim h_{F^{(i)}}^{s-2}$. Combining the two above estimates, we obtain $J_{2} \lesssim h_{K}^{s-2}\|u\|_{H^{s}(\omega_{K})}$. 
  Thus, we conclude $I_{2} \lesssim J_{1} + J_{2} \lesssim h_{K}^{s-2}\|u\|_{H^{s}(\omega_{K})}$. 
  
  Combining the estimate of $I_{1}$, $I_{2}$, and \eqref{eq:scale-phi} in Lemma \ref{lm:scale-phi-mu} (scale of $\phi_{i}$), we obtain $|\Pi_{h}u - w_{h}|_{H^{m}(K)} \lesssim h_{K}^{s-m}\|u\|_{H^{s}(\omega_{K})}$. 
  Then, \eqref{eq:app-Pi-ell-K} is obtained by triangle inequality and 
  Theorem \ref{thm:app_Vh} (approximation property of $\Pi_{h}$). \eqref{eq:app-Pi-ell-F} follows from Lemma \ref{lm:trace} (trace estimate) and \eqref{eq:app-Pi-ell-K}. The proof is thus complete.

\section{Proof of Lemma \ref{lm:poin}}\label{app:poincare}
Let $K$ be a curved triangle with vertices $P_{j}$, $j = 1,2,3$ (see Figure \ref{fig:curved-triangle}), $F = \wideparen{P_{2}P_{3}}$ be the curved edge. 

\begin{figure}[!htbp]
    \centering
    \captionsetup{justification=centering}
    \includegraphics[width=0.3\textwidth]{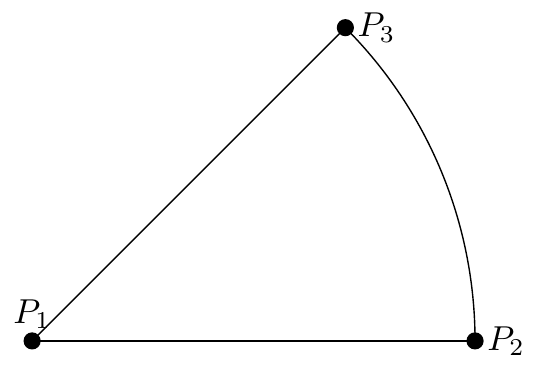}
    \caption{curved triangle $K$}
    \label{fig:curved-triangle}
\end{figure}

    Let $(z_{1}, z_{2})$ be the coordinates of point $P_{1}$. We have
    $$\operatorname{div}( \left(\begin{array}{c}
    x_{1} - z_{1} \\
    x_{2} - z_{2}
    \end{array}\right) u^{2} ) = 2 u^{2} + 2 u \left((x_{1} - z_{1})\partial_{1}u + (x_{2} - z_{2})\partial_{2}u\right).
    $$
    By integrating over $K$, we have 
    \begin{equation*} 
    \begin{aligned}
    &\quad \int_{K} u^{2} \mathrm{d}x \\
    &= \frac{1}{2} \int_{K}^{} \operatorname{div} ( \left(\begin{array}{c}
        x_{1} - z_{1} \\
        x_{2} - z_{2}
        \end{array}\right) u^{2} )\mathrm{d}x - \int_{K} u((x_{1} - z_{1})\partial_{1}u + (x_{2} - z_{2})\partial_{2}u )\mathrm{d}x \\ 
    & \leq \frac{1}{2} \int_{\partial K}^{} \boldsymbol{n} \cdot \left(\begin{array}{c}
        x_{1} - z_{1} \\
        x_{2} - z_{2}
        \end{array}\right) u^{2} \mathrm{d}s + \operatorname{diam}(K) \int_{K}^{}|u|(|\partial_{1}u| + |\partial_{2}u|)\mathrm{d}x. 
    \end{aligned}
    \end{equation*} 
    By Cauchy-Schwarz inequality, we get
    $$ \operatorname{diam}(K) \int_{K}^{}|u|(|\partial_{1}u| + |\partial_{2}u|)\mathrm{d}x \leq \frac{1}{2} \int_{K} u^{2} \mathrm{d}x + \operatorname{diam}(K)^{2} \int_{K} |\nabla u|^{2}\mathrm{d}x.$$ 
    Note that $\boldsymbol{n} \cdot (x_{1} - z_{1}, x_{2} - z_{2})^{T} \equiv  0$ on $\overline{P_{1}P_{2}}$ and $\overline{P_{3}P_{1}}$, thus 
    $$ \int_{\partial K}^{} \boldsymbol{n} \cdot \left(\begin{array}{c}
        x_{1} - z_{1} \\
        x_{2} - z_{2}
        \end{array}\right) u^{2} \mathrm{d}s = \int_{F}^{} \boldsymbol{n} \cdot \left(\begin{array}{c}
        x_{1} - z_{1} \\
        x_{2} - z_{2}
        \end{array}\right) u^{2} \mathrm{d}s \leq \operatorname{diam}(K) \int_{F}u^{2}\mathrm{d}s. $$
    In summary, we have
        \begin{equation*} 
        \begin{aligned}
        \int_{K}u^{2}\mathrm{d}x \leq  \frac{1}{2}\operatorname{diam}(K) \int_{F}u^{2}\mathrm{d}s + \frac{1}{2} \int_{K} u^{2} \mathrm{d}x + \operatorname{diam}(K)^{2} \int_{K} |\nabla u|^{2}\mathrm{d}x,
        \end{aligned}
        \end{equation*} 
        or
        $$\int_{K}u^{2}\mathrm{d}x \leq \operatorname{diam}(K) \int_{F}u^{2}\mathrm{d}s +  2\operatorname{diam} (K)^{2} \int_{K} |\nabla u|^{2}\mathrm{d}x. $$
        We then conclude the proof by applying \eqref{eq:diamK-hK}.

\section*{Acknowledgments}
The authors would like to express their gratitude to Prof. Jun Hu in Peking University for his helpful discussions. 

\bibliographystyle{siamplain}
\bibliography{oblique} 

\end{document}